%% file: Mazur-Tate.tex
\documentclass[a4paper]{amsart}
\usepackage{bbm}
\usepackage{color}
\usepackage{enumerate}
\usepackage{cite}
\usepackage[utf8]{inputenc}
\usepackage[all,cmtip]{xy}
\usepackage{etoolbox}
\apptocmd{\thebibliography}{\raggedright}{}{}

\usepackage{amscd}

\usepackage{tikz}
\usetikzlibrary{arrows}

\usepackage[leqno]{amsmath}
\usepackage{amsthm,amssymb,amsfonts}
\usepackage{wasysym}
\theoremstyle{plain}
\newtheorem{Theorem}{Theorem}[section]
\newtheorem{Lemma}[Theorem]{Lemma}

\newtheorem{Corollary}[Theorem]{Corollary}
\newtheorem{Proposition}[Theorem]{Proposition}
\newtheorem*{theorem}{Theorem}
\newtheorem*{conjecture}{Conjecture}
\theoremstyle{definition}
\newtheorem{Def}[Theorem]{Definition}
\newtheorem{Remark}[Theorem]{Remark}
\newtheorem*{defi}{Definition}

\numberwithin{equation}{section}

\newcommand{\Z}{\ensuremath{\mathbb{Z}}}

\newcommand{\OO}{\ensuremath{\mathcal{O}}}

\newcommand{\Q}{\ensuremath{\mathbb{Q}}}

\newcommand{\R}{\ensuremath{\mathbb{R}}}

\newcommand{\C}{\ensuremath{\mathbb{C}}}

\newcommand{\HH}{\ensuremath{\mathbb{H}}}

\newcommand{\p}{\ensuremath{\mathfrak{p}}}
\newcommand{\q}{\ensuremath{\mathfrak{q}}}

\newcommand{\A}{\ensuremath{\mathbb{A}}}
\newcommand{\I}{\ensuremath{\mathbb{I}}}

\newcommand{\Ah}{\ensuremath{\mathcal{A}}}

\newcommand{\inv}{^{-1}}									

\newcommand{\tild}[1]{\ensuremath{\widetilde{#1}}}						
\newcommand{\quot}[2]{{\raisebox{.2em}{$#1$}\left/\raisebox{-.2em}{$#2$}\right.}}		

\newcommand{\too}{\longrightarrow}								
\newcommand{\mapstoo}{\longmapsto}

\newcommand{\fii}{\ensuremath{\varphi}}								

\DeclareMathOperator{\Hom}{Hom}

\DeclareMathOperator{\Gal}{Gal}
\DeclareMathOperator{\Tr}{Tr}
\DeclareMathOperator{\invol}{inv}

\DeclareMathOperator{\Ima}{Im}

\DeclareMathOperator{\chr}{char}

\DeclareMathOperator{\GL}{GL}
\DeclareMathOperator{\PGL}{PGL}
\DeclareMathOperator{\SO}{SO}

\DeclareMathOperator{\rank}{rank}
\DeclareMathOperator{\ord}{ord}

\DeclareMathOperator{\rec}{rec}
\DeclareMathOperator{\Div}{Div}

\DeclareMathOperator{\Dist}{Dist}
\DeclareMathOperator{\St}{St}
\DeclareMathOperator{\cind}{c-ind}

\DeclareMathOperator{\sm}{sm}
\DeclareMathOperator{\sing}{sing}
\DeclareMathOperator{\fd}{fd}
\DeclareMathOperator{\hol}{hol}

\DeclareMathOperator{\tor}{tor}
\DeclareMathOperator{\MT}{MT}
\DeclareMathOperator{\cyc}{cyc}
\renewcommand{\det}{\operatorname{det}}
\renewcommand{\Dist}{\operatorname{Dist}}
\newcommand{\cf}{{\mathbbm 1}}
\newcommand{\into}{\hookrightarrow}
\newcommand{\onto}{\twoheadrightarrow}

\newcommand{\PP}{\ensuremath{\mathbb{P}^1}}					
\newcommand{\sinfty}{\ensuremath{^{S,\infty}}}
\newcommand{\G}{\ensuremath{\mathcal{G}}}
\newcommand{\n}{\ensuremath{\mathfrak{n}}}
\newcommand{\m}{\ensuremath{\mathfrak{m}}}
\DeclareMathOperator{\ev}{ev}
\DeclareMathOperator{\Ev}{Ev}
\DeclareMathOperator{\vol}{vol}

\newcommand{\dd}{\textnormal{d}}

\title[Order of vanishing of Hilbert modular Stickelberger elements]{On the order of vanishing of Stickelberger elements of Hilbert modular forms}

\author[F. Bergunde]{Felix Bergunde}
\address{F. Bergunde \\ Fakult\"at f\"ur Mathematik \\ Universit\"at Bielefeld \\ Universit\"atsstra\ss e 25 \\ 33615 Bielefeld \\ Germany}
\email{fbergund@math.uni-bielefeld.de}
\author[L. Gehrmann]{Lennart Gehrmann}
\address{L. Gehrmann \\ Fakult\"at f\"ur Mathematik \\ Universit\"at Duisburg-Essen \\ Thea-Leymann-Stra\ss e 9 \\ 45127 Essen \\ Germany}
\email{lennart.gehrmann@uni-due.de}

\begin{document}

\begin{abstract}
We construct Stickelberger elements for Hilbert modular cusp forms of parallel weight 2 and use recent results of Dasgupta and Spieß to bound their order of vanishing from below. As a special case the vanishing part of Mazur and Tate's refined ``Birch and Swinnerton-Dyer type''-conjecture for elliptic curves of rank 0 follows. 
\end{abstract}
\maketitle

\tableofcontents

\section*{Introduction}
\input{secintro.tex}

\section{Modular symbols}
\input{seccoh-distsp.tex}
\input{seccoh-hecke.tex}
\input{seccoh-cohom.tex}
\input{seccoh-delta.tex}
\input{seccoh-char.tex}
\input{seccoh-vanish.tex}

\section{Automorphic forms}
\input{secaut-autom.tex}
\input{secaut-specval.tex}
\input{secaut-lifting.tex}

\bibliographystyle{abbrv}
\bibliography{bibfile}
\end{document}

%% file: secintro.tex
Let $A$ be an elliptic curve over $\Q$ of conductor $N$.
By the modularity theorem one can associate a normalized newform $f \in S_2(\Gamma_0(N))$ to $A$ such that the corresponding $L$-series coincide (cf.~\cite{Wi}, \cite{TW} and \cite{BCDT}).
Let $\lambda_A \colon \Q/\Z \to \C$ be the modular symbol for $A$ given by $$\lambda_A(q) = 2 \pi i \int_{i\infty}^q f(z)dz$$ as defined in \cite{MTT}.

The N\'eron lattice $\mathfrak{L}_A$ of $A$ is obtained by integrating a N\'eron differential $\omega_A$ against all elements in $H_1(A(\C),\Z)$.
There exists a pair of positive real numbers $\Omega_A^+,\Omega_A^- \in \R_{>0}$ uniquely determined by the following property:
If $A(\R)$ has two connected components 
\begin{align*}
 \mathfrak{L}_A &=  \Omega_A^+ \Z + i \Omega_A^- \Z&& \textnormal{(the ``rectangular case'')}
\intertext{and if $A(\R)$ has only one connected component} 
 \mathfrak{L}_A &\subseteq \Omega_A^+ \Z + i \Omega_A^- \Z && \textnormal{(the ``nonrectangular case'')}
\end{align*}
with index two.
To be exact, in the nonrectangular case elements in $\mathfrak{L}_A$ are of the form $a \Omega_A^+ + i b \Omega_A^-$ with $a \equiv b \bmod 2$. 
By a beautiful theorem of Manin and Drinfeld (cf.~\cite{Ma} and \cite{Dr}) there exists a proper subring $\mathcal{R}\subset \Q$ such that $\lambda_A(q) = [q]_A^+ \Omega_A^+ + i [q]_A^- \Omega_A^-$ with functions $[\cdot]_A^\pm \colon \Q/\Z \to \mathcal{R}$.
For example, if $A$ is a strong Weil curve, we can  take $\mathcal{R}=\Z[\frac{1}{\tau c_{A}}]$, where $\tau$ is the order of the finite group $A(\Q)_\textnormal{tors}$ and $c_{A}$ is the Manin constant of $A$.
The Manin constant $c_{A}$ is an integer, which is conjectured to be $1$.
This is known in many cases (cf.~\cite{Ed}).
In special cases one has even better bounds for the denominators occuring in the ring $\mathcal{R}$ (cf.~\cite{Wu}).
The functions $[\cdot]_A^\pm$ are the so called ``$+$'' resp.~``$-$'' modular symbols.
In the following we just treat the ``$+$'' modular symbol, so we write $[\cdot]_A$ for $[\cdot]_A^+$.

Following Mazur and Tate's approach in \cite{MT} we define Stickelberger elements in this situation.
Fixing an integer $M\geq 1$ we write $\mu_{M}$ for the group of $M$-th roots of unity and $L=\Q(\mu_M)^+$ for the maximal totally real subextension of $\Q(\mu_M)$ and $G_M = \Gal(L/\Q)$.
We have an isomorphism $$(\Z/M\Z)^\ast/\{ \pm 1 \}  \stackrel{\cong}{\longrightarrow} G_M ,$$ where the image of $a \in (\Z/M\Z)^\ast$ is denoted by $\sigma_a$.
The Stickelberger element of modulus $M$ associated to $A$ is defined as
\begin{align*}
 \Theta^{\MT}_{A,M}= \frac{1}{2} \sum_{a\in (\Z/M\Z)^{\ast}} \left[\frac{a}{M}\right]_A \sigma_a \in \frac{1}{2}\mathcal{R}[G_M],
\end{align*}
where for an arbitrary (commutative and unital) ring $R$ and an arbitrary group $H$ the group algebra of $H$ over $R$ is denoted by $R[H]$.
Since $[q]_A=[-q]_A$ for all $q\in \Q/\Z$ (cf.~\cite{MTT}) we have
\begin{align*}
 \Theta^{\MT}_{A,M} \in \mathcal{R}[G_M]
\end{align*}
as long as $M\geq 3$, which we will assume from now on.

Next, we state the vanishing conjecture of Mazur and Tate. 
For general $R$ and $H$ as above, let $I_{R}(H)\subseteq R[H]$ be the kernel of the augmentation map $R[H]\to R$, $h\mapsto 1$.
\begin{defi}
 The order of vanishing $\ord_R(\xi)$ of an element $\xi \in R[H]$ is defined as
\begin{align*}
 \ord_R(\xi) = \begin{cases}
                r 	& \text{ if } \xi \in I_{R}(H)^r \setminus I_{R}(H)^{r+1}, \\
                \infty 	& \text{ if } \xi \in I_{R}^{r}(H)\ \forall r\geq 1.
               \end{cases}              
\end{align*}
\end{defi}
\noindent
Let $S_M$ be the set of prime factors $p$ of $M$ such that $A$ has split multiplicative reduction at $p$.
We define
\begin{align*}
 r_M = \rank A(\Q) + \left|S_M\right|.
\end{align*}

\begin{conjecture}[Mazur-Tate] \label{vanishing}
Let $\mathcal{R}\subset \Q$ be a subring, which contains $[q]_A$ for all $q\in \Q$.
Then the inequality $$\ord_\mathcal{R}(\Theta_{A,M}^{\MT}) \geq r_M$$ holds, i.e.~$\Theta_{A,M}^{\MT} \in I_\mathcal{R}\left(G_M\right)^{r_M}$.
\end{conjecture}

Our main objective is to proof the following

\begin{theorem}
There exists a free $\Z$-module $\mathcal{L}\subset \C$ of rank $1$ such that $[q]_A\in \mathcal{L}$ for all $q\in \Q$.
Given any such module $\mathcal{L}\subset\C$ we have $$\Theta_{A,M}^{\MT}\in I_\Z\left(G_M\right)^{|S_{m}|} \otimes \mathcal{L}.$$
In particular, if $\mathcal{R}\subset \Q$ is a subring which contains $[q]_A$ for all $q\in \Q$, we have
$$\ord_\mathcal{R}(\Theta_{A,M}^{\MT}) \geq \left|S_M\right|.$$
\end{theorem}

Thus the conjecture of Mazur and Tate is true if $\rank A(\Q) = 0$.
In fact, we prove a more general statement for Stickelberger elements coming from Hilbert modular cusp forms of parallel weight (2,\ldots,2). 

The structure of the paper is the following.
We fix once and for all a totally real number field $F$.
In Chapter 1 we give a description of modular symbols of trivial weight in terms of the group cohomology of $\PGL_{2}(F)$ with values in certain adelic function spaces.
For every finite abelian extension $L/F$ and every ideal $\mathfrak{m}$, which bounds the ramification of $L/F$, Stickelberger elements are constructed by first pulling back the modular symbol along the diagonal torus and then taking cap products with homology classes which are essentially given by the Artin reciprocity map.
We use results of Dasgupta and Spieß (cf.~\cite{DS}) to give a general criterion for obtaining lower bounds on the order of vanishing of these Stickelberger elements (see Proposition \ref{vanish2}).
Proposition \ref{compatibility} gives relations between Stickelberger elements for different moduli and field extensions, which generalize the corresponding statements for the Stickelberger elements of Mazur and Tate.
A functional equation for Stickelberger elements is given (see Proposition \ref{funceq}) and used to determine the parity of their order of vanishing in terms of local root numbers.
An advantage of our adelic approach is that all computations turn out to be completely local.
These computations are carried out in Section \ref{LocalHecke}.

In the second chapter we use the Eichler-Shimura homomorphism (or rather a version of it in our setting) to construct modular symbols for Hilbert modular cusp forms of parallel weight $(2,\ldots,2)$.
Thus, we can use the machinery developed in Chapter 1 to attach Stickelberger elements to Hilbert modular forms.
The fact that modular symbols commute with flat base change implies that these Stickelberger elements take values in integral group rings (see Lemma \ref{intStick}).
In Proposition \ref{specialvalues} we evaluate them at primitive characters and relate them to special values of $L$-functions.
Finally, we give lower bounds on their order of vanishing in terms of the number of all primes $\q$ such that the local component at $\q$ of the associated automorphic representation is Steinberg (see Theorem \ref{MainTheorem}).
This is an integral refinement of one of the main theorems of Spie\ss' article \cite{Sp} on the order of vanishing of $p$-adic $L$-functions of modular elliptic curves at $s=0$. More precisely, Spie\ss~ shows in \textit{loc.~cit.~}that the order of vanishing is at least the number of primes above $p$ at which the automorphic representation is Steinberg.

It should be mentioned that all arguments carry over (with slight modifications) to cohomological cuspidal automorphic representations of $\PGL_2$ of trivial cohomological weight over arbitrary number fields.
For ease of exposition we decided to stick to the totally real case.

\bigskip
\textbf{Acknowledgments:} We thank Michael Spie\ss~for suggesting to work on the Mazur-Tate conjecture and Thomas Poguntke for useful comments on an earlier draft.
The first named author acknowledges financial support provided by the DFG within the CRC 701 'Spectral Structures and Topological Methods in Mathematics'.

\bigskip
\textbf{Notions and Notations}. We will use the following notions and notations throughout the whole article.

All rings are commutative and unital.
The group of invertible elements of a ring $R$ will be denoted by $R^{\ast}$.
If $R$ is a ring and $H$ a group, we will denote the group algebra of $H$ over $R$ by $R[H]$.
We let $I_{R}(H)\subseteq R[H]$ be the kernel of the augmentation map $R[H]\to R$, $h\mapsto 1$.
If $\chi:H\to R^{\ast}$ is a group homomorphism, we let $R(\chi)$ be the representation of $H$ whose underlying $R$-module is $R$ itself and on which $H$ acts via the character $\chi$.
If $M$ is another $R[H]$-module, we put $M(\chi)=M\otimes_{R}R(\chi)$.

For a set $X$ and a subset $A\subseteq X$ the characteristic function $\cf_{A}\colon X\to \left\{0,1\right\}$ is defined by
\begin{align*}
 \cf_{A}(x) = \begin{cases}
		1 & \mbox{if } x\in A,\\
		0 & \mbox{else}.
	      \end{cases}
\end{align*}

Throughout the paper $F$ denotes a totally real number field of degree $d$ with ring of integers $\mathcal{O}=\mathcal{O}_{F}$.
Let $E=\mathcal{O}^{\ast}$ denote the group of global units.
For a nonzero ideal $\mathfrak{a} \subseteq \mathcal{O}$ we set $N(\mathfrak{a})=\sharp (\mathcal{O}/\mathfrak{a})$.

If $v$ is a place of $F$, we denote by $F_{v}$ the completion of $F$ at $v$.
If $\mathfrak{p}$ is a finite place, we let $\mathcal{O}_{\mathfrak{p}}$ denote the valuation ring of $F_{\mathfrak{p}}$ and $\ord_{\mathfrak{p}}$ the additive valuation such that $\ord_{\mathfrak{p}}(\varpi)=1$ for any local uniformizer $\varpi\in\mathcal{O}_{\mathfrak{p}}$.
For an arbitrary place $v$ let $|\cdot|_{v}$ be the normalized multiplicative norm, i.e.~$|x|_{\mathfrak{p}}=N(\mathfrak{p})^{-\ord_{p}(x)}$ if $\mathfrak{p}$ is a finite place and $|x|_{v}=|\sigma_{v}(x)|$ if $\sigma_{v}$ is the embedding $F\into\R$ corresponding to the Archimedean place $v$.
We denote by $U_{v}$ the invertible elements of $\mathcal{O}_{v}$ if $v$ is a finite place and the group of positive elements of $F_{v}$ if $v$ is a real place.
For a finite place $\mathfrak{p}$ we put $U^{(m)}_{\mathfrak{p}}=\left\{x\in U_{\mathfrak{p}}\mid x\equiv 1 \bmod \mathfrak{p}^{m}\right\}$. Further we denote by $F^{\ast}_{+}$ (resp.~$E_{+}$) the totally positive elements in $F$ (resp.~$E$).

Let $\A$ be the ring of adeles of $F$ and $\I$ the idele group of $F$.
We denote by $|\cdot|\colon\I\to\R^{\ast}$ the absolute modulus, i.e.~$|(x_{v})_{v}|=\prod_{v}{|x_{v}|_{v}}$ for $(x_{v})_{v}\in\I$.
For a finite set $S$ of places of $F$ we define the "$S$-truncated adeles" $\A^{S}$ (resp.~"$S$-truncated ideles" $\I^{S}$) as the restricted product of the completions $F_{v}$ (resp.~$F_{v}^{\ast}$) over all places $v$ which are not in $S$.
We put $F_{S}=\prod_{v\in S}{F_{v}}$, $U_{S}=\prod_{v\in S}U_{v}$ and $U^{S}=\prod_{v\notin S}U_{v}$.
The set of Archimedean places of $F$ will be denoted by $S_{\infty}$.
We often write $\A^{S,\infty}$ instead of $\A^{S\cup S_{\infty}}$ and similarly we define $\I^{S,\infty}$, $U^{S,\infty}$, $F_{S,\infty}$ etc.
We always drop the superscript $\emptyset$ if $S=\emptyset$.
Moreover, if $\widetilde{U}\subseteq U$ is a subgroup, we will write $\widetilde{U}^{S}$ for the image of $\widetilde{U}$ under the projection $U\to U^{S}$.
For every non-zero ideal $\mathfrak{m}=\prod_{\mathfrak{p}}\mathfrak{p}^{m_{\mathfrak{p}}}\subseteq\mathcal{O}$ we put
\begin{align*}
 U(\mathfrak{m})=\prod_{\mathfrak{p}\notin S_\infty}U_{\mathfrak{p}}^{(m_{\mathfrak{p}})}\times \prod_{v\in S_\infty}U_{v}.
\end{align*}

We will write $G$ for the algebraic group $\PGL_{2}$, $B$ for the Borel subgroup of upper triangular matrices and $T$ for the torus of diagonal matrices in $G$. The embedding
\begin{align*}
 \iota\colon \mathbb{G}_{m}\too G,\ x\mapstoo \begin{pmatrix}x&0\\0&1\end{pmatrix}
\end{align*}
induces an isomorphism of algebraic groups $\mathbb{G}_{m}\cong T$.
If $K\subset G(\A^{\infty})$ is a compact open subgroup and $S$ a finite set of finite places of $F$, we write $K^{S}$ for the image of $K$ under the projection $G(\A^{\infty})\to G(\A\sinfty)$.
We let $G(F_{\infty})^+\subset G(F_{\infty})$ be the subgroup of elements with totally positive determinant and similarly define $G(\A)^+\subset G(\A)$.
Given a subgroup $H\subset G(\A)$ we write $H^{+}$ for the intersection of $H$ with  $G(\A)^+$.

%% file: seccoh-distsp.tex
\subsection{Generalities on functions and distributions}
Given topological spaces $X,Y$ we will write $C(X,Y)$ for the space of continuous functions from $X$ to $Y$.
If $R$ is a topological ring, we define $C_{c}(X,R) \subseteq C(X,R)$ as the subspace of continuous functions with compact support.
If we consider $Y$ (resp.~$R$) with the discrete topology, we will often write $C^{0}(X,Y)$ (resp.~$C_{c}^{0}(X,R)$) instead.

Since a locally constant map with compact support takes only finitely many different values, the canonical map
\begin{align*}
 C^{0}_{c}(X,\Z)\otimes R\too C^{0}_{c}(X,R)
\end{align*}
is an isomorphism of $R$-modules.
Given two topological spaces $X$ and $Y$, the pairing
\begin{align*}
 C^{0}_{c}(X,R)\times C^{0}_{c}(Y,R)\too C^{0}_{c}(X\times Y,R),\ (f,g)\mapstoo f\cdot g
\end{align*}
induces an isomorphism
\begin{align}\label{product}
 C^{0}_{c}(X,R)\otimes_{R} C^{0}_{c}(Y,R)\cong C^{0}_{c}(X\times Y,R).
\end{align}
For a ring $R$ and an $R$-module $N$, we define the $R$-module of $N$-valued distributions on $X$ as $\Dist(X,N)=\Hom_{\Z}(C^{0}_{c}(X,\Z),N)$.
Using (\ref{product}) we get isomorphisms
\begin{align}\label{product2}
 \Dist(X \times Y,N) \cong \Dist(X,\Dist(Y,N))\cong \Dist (Y,\Dist(X,N)).
\end{align}
If $X$ is discrete, we have the following pairing
\begin{align*}
 C_{c}^{0}(X,\Z)\times C^{0}(X,N)\too N,\ (\psi,\phi)\mapstoo \sum_{x\in X} (\psi\cdot\phi)(x),
\end{align*}
which induces an isomorphism of $R$-modules
\begin{align}\label{discrete}
 C^{0}(X,N)\too \Dist(X,N).
\end{align}
If $f \colon X \to Y$ is an open embedding,
extension by zero
\begin{align*}
 f_! \colon C_{c}^{0}(X,R)\too C_{c}^{0}(Y,R)
\end{align*}
induces a map on distribution spaces
\begin{align}\label{opensub}
 f^! \colon \Dist(Y,N)\too\Dist(X,N).
\end{align}
Let $f\colon X \to Y$ be a map between discrete spaces.
The homomorphism
\begin{align*}
 f_{\ast}\colon C_{c}(X,R)\too C_{c}(Y,R),\ \Phi\mapstoo \left[y\mapsto \sum_{f(x)=y}\Phi(x)\right]
\end{align*}
induces an $R$-linear map of distribution spaces
\begin{align}\label{discrete2}
 f^{\ast}\colon\Dist(Y,N)\too \Dist(X,N),
\end{align}
which corresponds to the pullback map on $N$-valued functions under the isomorphism (\ref{discrete}).

On the other hand, if $g\colon X\to Y$ is a proper, continuous map, the pullback
\begin{align*}
 g^{\ast}\colon C_{c}^{0}(Y,R)\too C_{c}^{0}(X,R),\ \phi \mapstoo \phi \circ g
\end{align*}
along $g$ always yields an $R$-linear map
\begin{align}\label{properpush}
 g_{\ast}\colon \Dist(X,N)\too \Dist(Y,N).
\end{align}
If $g\colon X\to Y$ is a proper map between discrete spaces, this corresponds to the map
\begin{align*}
 g_{\ast}\colon C(X,N)\too C(Y,N),\ \Phi\mapstoo \left[y\mapsto \sum_{f(x)=y}\Phi(x)\right]
\end{align*}
under the isomorphism (\ref{discrete}).

Now, let $H$ be a topological group and $K$ a closed subgroup.
Then $H$ acts on $C^{0}(H/K,N)$ (resp.~$C^{0}_{c}(H/K,R)$) via left-multiplication, i.e.~$(h.f)(x)=f(h^{-1}x)$.
Thus, we also get an $H$-action on $\Dist(H/K,N)$ via $(h.D)(f)=D(h^{-1}f)$.
Suppose $K\subseteq H$ is open.
Then $H/K$ is discrete and it follows immediately that the isomorphism (\ref{discrete}) is $G$-equivariant.

%% file: seccoh-hecke.tex
\subsection{Hecke operators and compatibility} \label{LocalHecke}
This section contains all local computations, which are needed for the construction of Stickelberger elements of modular symbols in later sections.
In particular, Lemma \ref{LocalHeckeLemma} is the key lemma for proving the compatibility of Stickelberger elements of different moduli and Lemma \ref{LocalDiagram} enables us to apply the formalism of Dasgupta and Spie\ss~to our situation.

Let $\p$ be a fixed finite place of $F,\ q=N(\p)$ and $\varpi_\p$ a local uniformizer at $\p$.
We choose representatives $a_i \in \OO_\p$, $i=0,\dots,q-1$, of $\OO_\p/\p$ and define
\begin{align*}
 \gamma_i=\begin{pmatrix} \varpi_\p & a_i \\ 0 & 1 \end{pmatrix} \mbox{for}\ i=0,\dots,q-1,\ \mbox{ and }
 \gamma_q=\begin{pmatrix} 1 & 0 \\ 0 & \varpi_\p \end{pmatrix}.
\end{align*}
For simplicity, we will always assume that $a_0=0$.
For an integer $n \geq 0$ we define
\begin{align*}
 K_\p(\p^n) = \left\{ \begin{pmatrix} a & b \\ c & d \end{pmatrix} \in G(\OO_\p) \mid c \equiv 0 \bmod \p^{n} \right\}
\end{align*}
and consider the well known double-coset-decomposition
\begin{align*}
 K_\p(\p^n) \begin{pmatrix} \varpi_\p & 0 \\ 0 & 1 \end{pmatrix} K_\p(\p^n) = \begin{cases}
										  \bigcup_{i=0}^{q-1} \gamma_i K_\p(\p^n) \cup \gamma_q K_\p(\p^n)	& \mbox{if $n=0$,} \\
										  \bigcup_{i=0}^{q-1} \gamma_i K_\p(\p^n)				& \mbox{else.}
									      \end{cases}
\end{align*}
Let us fix a ring $R$ and an $R$-module $N$.
Given $g \in G(F_\p)$ and a compact open subgroup $K_\p \subseteq G(F_\p)$ we let $$\rho(g) \colon C(G(F_\p)/K_\p,N) \too C(G(F_\p)/ (g K_\p g\inv),N)$$ be the map induced by right-multiplication, i.e.~$\rho(g)(\phi)(\ast) = \phi(\ast \cdot g)$.
Since right- and left-multiplication commute, we see that $\rho(g)$ is $G(F_\p)$-equivariant.

\begin{Def} 
Let $n \geq 0$ be an integer.
The $\p$-Hecke operator $$T_\p \colon C(G(F_\p)/K_\p(\p^n),N) \to C(G(F_\p)/K_\p(\p^n),N)$$ is given by
\begin{align*}
 T_\p = \sum_{i=0}^{q-1} \rho(\gamma_i) + \cf_\p(\p^n) \rho(\gamma_q) 	\mbox{, where } \cf_\p(\p^n)=\begin{cases}
															1 & \mbox{if $n=0$,}\\
															0 & \mbox{else.}
														     \end{cases}
\end{align*}
The Atkin-Lehner involution at $\p$ is defined as
\begin{align*}
 W_{\p^{n}}=\rho\left(\begin{pmatrix} 0 & 1 \\ \varpi_{\p}^{n} & 0 \end{pmatrix} \right)\colon C(G(F_\p)/K_\p(\p^n),N) \to C(G(F_\p)/K_\p(\p^n),N).
\end{align*}
\end{Def}

Note that the square of the matrix defining $W_{\p^n}$ is trivial in $G(F_\p)$ and hence, $W_{\p^n}$ defines an involution.

For integers $n,r \geq 0$ let
\begin{align*}
 \partial_r \colon C(G(F_\p)/K_\p(\p^n),N) \too \Dist(F_\p^\ast/U_\p^{(r)},N)
\end{align*}
be the map given by
\begin{align*}
 \partial_r(\phi)\left(\cf_{xU_\p^{(r)}}\right) = \phi\left(\begin{pmatrix} x & 0 \\ 0 & 1 \end{pmatrix} \begin{pmatrix} \varpi_\p^r & 1 \\ 0 & 1 \end{pmatrix} \right).
\end{align*}
Keep in mind that the set of functions $\cf_{xU_\p^{(r)}}$, where $x$ ranges over representatives of $F_\p^\ast$ modulo $U_\p^{(r)}$, is a $\Z$-basis of $C_c^0(F_\p^\ast/U_\p^{(r)},\Z)$.
The map $\partial_r$ is well-defined for all $r \geq 0$ since
\begin{align*}
 \begin{pmatrix} u & 0 \\ 0 & 1 \end{pmatrix} \begin{pmatrix} \varpi_\p^r & 1 \\ 0 & 1 \end{pmatrix} = \begin{pmatrix} \varpi_\p^r & 1 \\ 0 & 1 \end{pmatrix} \begin{pmatrix} u & \frac{u-1}{\varpi_\p^r} \\ 0 & 1 \end{pmatrix}
\end{align*}
and $\begin{pmatrix} u & \frac{u-1}{\varpi_\p^r} \\ 0 & 1 \end{pmatrix} \in K_\p(\p^n)$ for $u \in U_\p^{(r)}$.
The embedding $\iota\colon F_\p^{\ast}\to G(F_{p})$ induces an $F_{\p}^{\ast}$-action on $C(G(F_\p)/K_\p(\p^n),N)$.
The map $\partial_{r}$ is obviously $F_{\p}^{\ast}$-equivariant.
Note that $\begin{pmatrix} 1 & 1 \\ 0 & 1 \end{pmatrix} \in K_\p(\p^n)$ for all $n$.
Hence, we get $$\partial_0(\phi)\left(\cf_{xU_\p}\right) = \phi\left(\begin{pmatrix} x & 0 \\ 0 & 1 \end{pmatrix} \right).$$

By \eqref{discrete2} and \eqref{properpush}, the projection $\pi_r \colon F_\p^\ast/U_\p^{(r+1)} \to F_\p^\ast/U_\p^{(r)}$ yields maps
\begin{align*}
 (\pi_r)^\ast \colon \Dist(F_\p^\ast/U_\p^{(r)},N) \to \Dist(F_\p^\ast/U_\p^{(r+1)},N)
\intertext{and}
 (\pi_r)_\ast \colon \Dist(F_\p^\ast/U_\p^{(r+1)},N) \to \Dist(F_\p^\ast/U_\p^{(r)},N).
\end{align*}

\begin{Lemma} \label{LocalHeckeLemma}
Let $n\geq 0$ be an integer.
Then
\begin{align*}
 \partial_0(T_\p \phi) &= \varpi_\p\inv. \partial_0(\phi) + (\pi_0)_\ast(\partial_1(\phi)) + \cf_\p(\p^n) \varpi_\p. \partial_0(\phi)
\intertext{and if $r>0$, the equality}
 \partial_r(T_\p \phi) &= (\pi_{r})_\ast(\partial_{r+1}(\phi))  + \cf_\p(\p^n) (\pi_{r-1})^\ast(\partial_{r-1}(\phi))
\end{align*}
holds.
\end{Lemma}

\begin{proof}
Firstly, let us treat the extra summand for $n=0$.
If $r=0$, we have
\begin{align*}
 \partial_0(\rho(\gamma_q)\phi)\left(\cf_{x U_\p}\right) &= \phi\left(\begin{pmatrix} x & 0 \\ 0 & 1 \end{pmatrix} \begin{pmatrix} 1 & 0 \\ 0 & \varpi_\p \end{pmatrix}\right) \\
							 &= \phi\left(\begin{pmatrix} \varpi_\p\inv x & 0 \\ 0 & 1 \end{pmatrix}\right) \\
							 &= \varpi_\p. \partial_0(\phi)\left(\cf_{x  U_\p}\right).
\end{align*}
Note that the second equality holds since we are working in $\PGL_2$.
However, if $r>0$, we have
\begin{align*}
 \partial_r(\rho(\gamma_q)\phi)\left(\cf_{x U_\p^{(r)}}\right)  &= \phi(\left(\begin{pmatrix} x & 0 \\ 0 & 1 \end{pmatrix} \begin{pmatrix} \varpi_\p^{r} & 1 \\ 0 & 1 \end{pmatrix} \begin{pmatrix}  1 & 0 \\ 0 & \varpi_\p \end{pmatrix}\right) \\
								&= \phi(\left(\begin{pmatrix} x & 0 \\ 0 & 1 \end{pmatrix} \begin{pmatrix} \varpi_\p^{r} & \varpi_\p \\ 0 & \varpi_\p \end{pmatrix} \right) \\
								&= \phi(\left(\begin{pmatrix} x & 0 \\ 0 & 1 \end{pmatrix} \begin{pmatrix} \varpi_\p^{r-1} & 1 \\ 0 & 1 \end{pmatrix} \right) \\
								&= \partial_{r-1}(\phi)\left(\cf_{x U_\p^{(r-1)}}\right).
\end{align*}
It remains to take care of the other summands.
Thus, without loss of generality, it is enough to consider the case $n>0$.
For $r=0$ we have
\begin{align*}
 \partial_0(T_\p \phi)\left( \cf_{x U_\p} \right) = \partial_0(\phi)\left(\gamma_0 \cf_{x U_\p}\right) + \sum_{i=1}^{q-1} \partial_0(\rho(\gamma_i) \phi)\left(\cf_{x U_\p}\right)
\end{align*}
by definition.
For the first summand $$\partial_0(\rho(\gamma_0) \phi)\left(\cf_{x U_\p}\right) = \phi\left(\begin{pmatrix} \varpi_\p x & 0 \\ 0 & 1 \end{pmatrix}\right) = \varpi_\p\inv. \partial_0(\phi)\left(\cf_{x U_\p}\right)$$ holds.
Due to the fact that $a_i \in \OO^\ast$ for $1 \leq i \leq q-1$, we have $\begin{pmatrix}  a_i & 0 \\ 0 & 1 \end{pmatrix} \in K_\p(\p^n)$ and hence the second summand equals
\begin{align*}
 \sum_{i=1}^{q-1}\phi\left( \begin{pmatrix} x & 0 \\ 0 & 1 \end{pmatrix} \begin{pmatrix} \varpi_\p  & a_i \\ 0 & 1 \end{pmatrix} \right)
    &= \sum_{i=1}^{q-1}\phi\left( \begin{pmatrix} a_i x & 0 \\ 0 & 1 \end{pmatrix} \begin{pmatrix} \varpi_\p  & 1 \\ 0 & 1 \end{pmatrix} \begin{pmatrix} a_i\inv & 0 \\ 0 & 1 \end{pmatrix} \right) \\
    &= \sum_{i=1}^{q-1}a_{i}^{-1}.\phi\left( \begin{pmatrix} x & 0 \\ 0 & 1 \end{pmatrix} \begin{pmatrix} \varpi_\p  & 1 \\ 0 & 1 \end{pmatrix} \right)\\
		&= (\pi_0)_\ast(\partial_1(\phi))\left(\cf_{x U_\p}\right).
\end{align*}
For the remaining case, i.e.~$r,n>0$, we put $u_i = 1 + \varpi_\p^{r} a_i$ for $i=0,\dots,q-1$.
We get
\begin{align*}
  \partial_r(\rho(\gamma_i) \phi)\left(\cf_{x U_\p^{(r)}} \right) 
      &= \phi\left( \begin{pmatrix} x & 0 \\ 0 & 1 \end{pmatrix} \begin{pmatrix} \varpi_\p^r & 1 \\ 0 & 1 \end{pmatrix} \begin{pmatrix} \varpi_\p  & a_i \\ 0 & 1 \end{pmatrix} \right) \\
      &= \phi\left( \begin{pmatrix} x & 0 \\ 0 & 1 \end{pmatrix} \begin{pmatrix} \varpi_\p^{r+1} & u_i \\ 0 & 1 \end{pmatrix} \right) \\
      &= \phi\left( \begin{pmatrix} x & 0 \\ 0 & 1 \end{pmatrix} \begin{pmatrix} u_i & 0 \\ 0 & 1 \end{pmatrix} \begin{pmatrix} \varpi_\p^{r+1} & 1 \\ 0 & 1 \end{pmatrix} \begin{pmatrix} u_i\inv & 0 \\ 0 & 1 \end{pmatrix} \right) \\
      &= u_i^{-1}. \partial_{r+1}(\phi)\left(\cf_{x U_\p^{(r+1)}} \right)
\end{align*}
because $\begin{pmatrix} u_i & 0 \\ 0 & 1 \end{pmatrix}$ is an element of $K_\p(\p^n)$.
Since $\{ u_i \mid i=0,\dots,q-1 \}$ is a set of representatives of $U_\p^{(r)}/U_\p^{(r+1)}$, the claim follows.
\end{proof}

For a ring $R$ we define the Steinberg representation $\St_\p(R)$ to be the space of locally constant $R$-valued functions on $\PP(F_\p)$ modulo constant functions.
The group $G(F_\p)$ acts on $\PP(F_\p)$ by linear fractional transformations and hence, it acts on $\St_\p(R)$ via $(\gamma.f)(z)=f(\gamma \inv z)$ for $\gamma \in G(F_\p)$ and $f \in \St_\p(R)$.
If $R=\Z$, we will simply write $\St_\p$ instead of $\St_\p(R)$.
Extension by zero gives a map
\begin{align*}
 \delta_\p \colon C_c(F_\p,\Z) \too \St_\p,
\end{align*}
i.e.~$\delta_\p(\fii)([x:1])=\fii(x)$ and $\delta_\p(\fii)([1:0])=0$. 
This map is in fact bijective.
The inverse is given by 
\begin{align*}
 \St_\p \too C_c(F_\p,\Z),\  f \mapstoo \fii_f \mbox{, where } \fii_f(x) = f([x:1])-f([1:0]).
\end{align*}
The Borel subgroup $B(F_\p)$ acts on $C_c(F_\p,\Z)$ via
\begin{align*}
 \left(\begin{pmatrix} a & x \\ 0 & 1 \end{pmatrix} . \fii\right)(z)=\fii(a\inv(z-x)),
\end{align*}
where $\begin{pmatrix} a & x \\ 0 & 1 \end{pmatrix} \in B(F_\p)$, $\fii \in C_c(F_\p,\Z)$ and $z \in F_\p$.
The following lemma is a simple calculation.

\begin{Lemma}\label{equivariance}
The map $\delta_\p$ is $B(F_\p)$-equivariant.
\end{Lemma}

Let us fix the element $\fii_\p=\cf_{\OO_\p}$ of $\St_\p$.
It is easy to see that $\fii_\p$ is invariant under the Iwahori subgroup $K_\p(\p)$.
Hence, for every $R$-module $N$ evaluation at $\fii_\p$ defines a map
\begin{align*}
 \ev_\p \colon \Hom(\St_\p,N) \too C(G(F_\p)/K_\p(\p), N),\ \ev_\p(\phi)(g)=\phi(g \fii_\p).
\end{align*}
Let $r \geq 0$ be an integer.
We define $$\eta_r \colon \Dist(F_\p,N) \too \Dist(F_\p^\ast/U_\p^{(r)},N)$$ as the composition of the two maps
\begin{align*}
 \Dist(F_\p,N) \stackrel{\eqref{opensub}}{\too} \Dist(F_\p^\ast,N)
\end{align*}
and
\begin{align*}
 \Dist(F_\p^\ast,N) \stackrel{\eqref{properpush}}{\too} \Dist(F_\p^\ast/U_\p^{(r)},N).
\end{align*}

\begin{Lemma}\label{LocalDiagram}
Let $r \geq 1$ be an integer.
Then the following diagram is commutative:
  \begin{center}
  \begin{tikzpicture}
    \path 	(0,0) 	node[name=A]{$\Hom(\St_\p,N)$}
		(5,0) 	node[name=B]{$C(G(F_\p)/K_\p(\p),N)$}
		(0,-2) 	node[name=C]{$\Dist(F_\p,N)$}
		(5,-2) 	node[name=D]{$\Dist(F_\p^\ast/U_\p^{(r)},N)$};
    \draw[->] (A) -- (B) node[midway, above]{$\ev_\p$};
    \draw[->] (A) -- (C) node[midway, left]{$\delta_\p^\ast$};
    \draw[->] (B) -- (D) node[midway, right]{$\partial_r$};
    \draw[->] (C) -- (D) node[midway, below]{$\eta_r$};
  \end{tikzpicture} 
  \end{center}
\end{Lemma}

\begin{proof}
We have
\begin{align*}
 \partial_r(\ev_\p(\phi))\left(\cf_{xU_\p^{(r)}}\right) 	&= \phi\left(\begin{pmatrix} x & 0 \\ 0 & 1 \end{pmatrix} \begin{pmatrix} \varpi_\p^r & 1 \\ 0 & 1 \end{pmatrix} \cf_{\OO_\p}\right) \\
								&= \phi\left(\cf_{xU_\p^{(r)}}\right) \\
								&= \eta_r(\delta_\p^\ast(\phi))\left(\cf_{xU_\p^{(r)}}\right),
\end{align*}
where the last equality holds by definition.
\end{proof}

Finally, we want to see how the maps $\partial_{r}$ and the Atkin-Lehner involution intertwine.
The functional equation for Stickelberger elements (see Proposition \ref{funceq}) will be a direct consequence of this.
Let
\begin{align*}
 \invol \colon \Dist(F_\p^\ast/U_\p^{(r)},N) \too \Dist(F_\p^\ast/U_\p^{(r)},N)
\end{align*}
be the map induced by inversion (i.e.~the map which sends $x$ to $x^{-1}$).
\begin{Lemma}\label{localfuneq} \renewcommand{\labelenumi}{(\roman{enumi})}
Let $n\geq 0$ be an integer.
\begin{enumerate}
\item  The equality
\begin{align*}
 \partial_{0}(W_{\p^n}\phi)=\varpi_{\p}^{n}. \invol \left(\partial_{0}\left(\begin{pmatrix} 0 & 1 \\ 1 & 0 \end{pmatrix}.\phi \right)\right)
\end{align*}
holds for all $\phi\in C(G(F_\p)/K_\p(\p^n),N)$.
\item Assume that $r\geq n$.
Then
\begin{align*}
 \partial_{r}(\phi)=\invol \left(\partial_{r}\left(\begin{pmatrix} 0 & 1 \\ 1 & 0 \end{pmatrix}.\phi \right)\right)
\end{align*}
holds for all $\phi\in C(G(F_\p)/K_\p(\p^n),N)$.
\end{enumerate}
\end{Lemma}

\begin{proof}
(i) Using the fact that we are working in $\PGL_{2}$ we get
\begin{align*}
 \partial_{0}(W_{\p^n}\phi)(\cf_{xU_{\p}})
    &= \phi \left(\begin{pmatrix} x & 0 \\ 0 & 1 \end{pmatrix}\begin{pmatrix} 0 & 1 \\ \varpi_{\p}^{n} & 0 \end{pmatrix}\right)\\
    &= \phi \left(\begin{pmatrix} 0 & 1 \\ 1 & 0 \end{pmatrix}\begin{pmatrix} \varpi_{\p}^{n} & 0 \\ 0 & x \end{pmatrix}\right)\\
    &= \phi \left(\begin{pmatrix} 0 & 1 \\ 1 & 0 \end{pmatrix}\begin{pmatrix} \varpi_{\p}^{n} x^{-1} & 0 \\ 0 & 1 \end{pmatrix}\right)\\
    &= \varpi_{\p}^{n}.\invol \left(\partial_{0}\left(\begin{pmatrix} 0 & 1 \\ 1 & 0 \end{pmatrix}.\phi \right)\right)(\cf_{xU_{\p}}).
\end{align*}
(ii) Since
\begin{align*}
 \begin{pmatrix} 0 & 1 \\ 1 & 0 \end{pmatrix}\begin{pmatrix} \varpi^{r}_{\p} & 1 \\ 0 & 1 \end{pmatrix} =\begin{pmatrix} 0 & 1 \\ \varpi^{r}_{\p} & 1 \end{pmatrix} =\begin{pmatrix} \varpi^{r}_{\p} & 1 \\ 0 & 1 \end{pmatrix}\begin{pmatrix} -1 & 0 \\ \varpi^{r}_{\p} & 1 \end{pmatrix}
\end{align*}
holds and by assumption
\begin{align*}
 \begin{pmatrix} -1 & 0 \\ \varpi^{r}_{\p} & 1 \end{pmatrix} \in K_{\p}(\p^{n})
\end{align*}
we have
\begin{align*}
 \partial_{r}(\phi)(\cf_{xU_{\p}})
    &= \phi \left(\begin{pmatrix} x & 0 \\ 0 & 1 \end{pmatrix}\begin{pmatrix} \varpi^{r}_{\p} & 1 \\ 0 & 1 \end{pmatrix}\right)\\
    &= \phi \left(\begin{pmatrix} x & 0 \\ 0 & 1 \end{pmatrix}\begin{pmatrix} \varpi^{r}_{\p} & 1 \\ 0 & 1 \end{pmatrix}\begin{pmatrix} -1 & 0 \\ \varpi^{r}_{\p} & 1 \end{pmatrix}\right)\\
    &= \phi \left(\begin{pmatrix} x & 0 \\ 0 & 1 \end{pmatrix}\begin{pmatrix} 0 & 1 \\ 1 & 0 \end{pmatrix}\begin{pmatrix} \varpi^{r}_{\p} & 1 \\ 0 & 1 \end{pmatrix}\right)\\
    &= \phi \left(\begin{pmatrix} 0 & 1 \\ 1 & 0 \end{pmatrix}\begin{pmatrix} x^{-1} & 0 \\ 0 & 1 \end{pmatrix}\begin{pmatrix} \varpi^{r}_{\p} & 1 \\ 0 & 1 \end{pmatrix}\right)\\
    &=\invol \left( \partial_{r}\left(\begin{pmatrix} 0 & 1 \\ 1 & 0 \end{pmatrix}.\phi \right)\right)(\cf_{xU_{\p}}),
\end{align*}
which proves the claim.
\end{proof}

%% file: seccoh-cohom.tex
\subsection{Cohomology of $\PGL_2(F)$} \label{cohomology}
We introduce modular symbols for $\PGL_2(F)$ in terms of group cohomology.
Our treatment is similar to the one of Spie\ss~in \cite{Sp}.

Let $\Div(\PP(F))$ be the free abelian group over $\PP(F)$ and let $\Div_0(\PP(F))$ be the kernel of the map $$\Div(\PP(F)) \to \Z,\ \sum_P m_P P \mapsto \sum_P m_P.$$
Note that we have a $G(F)$-action on $\Div_0(\PP(F))$ induced by the $G(F)$-action on $\PP(F)$.

Given a ring $R$ and a finite set $S$ of finite places of $F$, we define the semi-local Steinberg representation at $S$ as $$\St_S(R)=\otimes_{\q \in S} \St_\q(R).$$
As in the local case, if $R=\Z$, we will simply write $\St_S$ instead of $\St_S(R)$.
For an $R$-module $N$ and a compact open subgroup $K \subseteq G(\A^\infty)$ we define
\begin{align}\label{Ah}
 \Ah(K,S;N) &= C(G(\A\sinfty)/K^{S},\Hom_\Z(\St_S,\Hom_\Z(\Div_0(\PP(F)),N))).
\end{align}
The $G(F)$-action on $\Ah(K,S;N)$ is given by $(\gamma.\Phi)(g)=\gamma \Phi(\gamma\inv g)$, where $\Phi \in \Ah(K,S;N)$, $\gamma \in G(F)$ and $g \in G(\A\sinfty)$.

We now fix locally constant homomorphisms $\epsilon_v \colon F_v^\ast \to \{\pm 1\} = \Z^\ast$ for every Archimedean place $v$ and define $\epsilon \colon F_\infty^\ast \to \{\pm 1\}$ via $$\epsilon((x_v)_{v \in S_\infty}) = \prod_{v \in S_\infty} \epsilon_v(x_v).$$
By abuse of notation we will also write $\epsilon$ for the homomorphism $G(F_\infty) \xrightarrow{\det} F_\infty^\ast \xrightarrow{\epsilon} \{ \pm 1 \}$.
We are interested in the cohomology of the $G(F)$-modules $\Ah(K,S;N)(\epsilon)$.

\begin{Proposition}\label{FlachundNoethersch} \renewcommand{\labelenumi}{(\roman{enumi})}
Let $S$ be a finite set of finite places of $F$ and $K \subseteq G(\A\sinfty)$ a compact open subgroup.
\begin{enumerate}
\item Let $N$ be a flat $R$-module equipped with the trivial $G(F)$-action. Then the canonical map
\begin{align*}
 H^q(G(F),\Ah(K,S;R)(\epsilon)) \otimes_R N \to H^q(G(F),\Ah(K,S;N)(\epsilon))
\end{align*}
is an isomorphism for all $q \geq 0$.
\item If $R$ is Noetherian, then the groups $H^q(G(F),\Ah(K,S;R)(\epsilon))$ are finitely generated $R$-modules for all $q \geq 0$.
\end{enumerate}
\end{Proposition}

\begin{proof}
This is almost verbatim Proposition 4.6. of \cite{Sp}.
\end{proof}

\begin{Def}
The module of $N$-valued, $S$-special modular symbols of $F$ of level $K$ and sign $\epsilon$ is given by $$\mathcal{M}(K,S;N)^{\epsilon} = H^{d-1}(G(F),\Ah(K,S;N)(\epsilon)).$$
If $S$ is the empty set, we will write $\mathcal{M}(K;N)^{\epsilon}$ instead of $\mathcal{M}(K,\emptyset;N)^{\epsilon}$.
\end{Def}

As in the local setting, given $g \in G(\A\sinfty)$ and a compact open subgroup $K \subseteq G(\A^\infty)$ we let $$\mathcal{R}(g) \colon \Ah(K,S;N)(\epsilon) \too \Ah(g\inv K g, S;N)(\epsilon)$$ be the map given by right-multiplication.
Since this map is $G(F)$-equivariant it induces a map on cohomology groups
$$\mathcal{R}(g)\colon \mathcal{M}(K,S;N)^{\epsilon} \too \mathcal{M}(gKg^{-1},S;N)^{\epsilon}.$$

Let $\mathfrak{n}\subseteq\mathcal{O}$ be a non-zero ideal.
We are mainly interested in the case that the level is given by $K=K_0(\n)=\prod_\p K_\p(\n)$ with $K_\p(\n)=K_\p(\n \OO_\p)$.
Since $\mathcal{M}(K_{0}(\mathfrak{n}),S;N)^\epsilon=\mathcal{M}(K_{0}(\mathfrak{n}\mathfrak{q}),S;N)^{\epsilon}$ for $\mathfrak{q} \in S$ we will always assume without loss of generality that every $\mathfrak{q}\in S$ divides $\mathfrak{n}$ exactly once.
Using the same notation as in Section \ref{LocalHecke} (and viewing $G(F_{\p})$ as a subgroup of $G(\A^{\infty})$ via the canonical embedding) we define for every finite place $\mathfrak{p}\notin S$ the global $\p$-Hecke operator $$T_\p \colon \mathcal{M}(K_0(\n),S;N)^{\epsilon} \to \mathcal{M}(K_0(\n),S;N)^{\epsilon}$$ via
\begin{align*}
 T_\p = \sum_{i=0}^{q-1} \mathcal{R}(\gamma_i) + \cf_\p(\n) \mathcal{R}(\gamma_q) \mbox{, where } \cf_\p(\n)= \cf_\p(\n \OO_\p).
\end{align*}
We fix the element $\fii_{S}=\otimes_{\mathfrak{q}\in S}\fii_{\mathfrak{q}}=\otimes_{\mathfrak{q}\in S}\cf_{\mathcal{O}_{\mathfrak{q}}}$ in $\St_{S}$.
Evaluation at $\fii_{S}$ induces a map on cohomology groups $$\Ev_{S}\colon \mathcal{M}(K_0(\n),S;N)^{\epsilon} \too \mathcal{M}(K_0(\n);N)^{\epsilon},$$ which commutes with the action of the Hecke operators $T_{\p}$ for $\p\notin S$.
We decompose $\n$ into a product $\n_{1}\cdot\n_2$ of coprime ideals and define the element $w_{\n_1} \in G(\A^\infty)$ by
\begin{align*}
 (w_{\n_1})_\p = \begin{cases}
                  \begin{pmatrix}0&1\\ \varpi_\p^{\ord_{\p}(\n_1)}&0\end{pmatrix} & \mbox{if $\p \mid \n_1$,}\\
		  \begin{pmatrix} 1 & 0 \\ 0 & 1 \end{pmatrix} & \mbox{else,}
		 \end{cases} 
\end{align*}
where $\varpi_\p$ is a local uniformizer at $\p$.
The global Atkin-Lehner $W_{\n_1}$ is given by
\begin{align*}
 W_{\n_1}=\mathcal{R}(w_{\n_1})\colon \mathcal{M}(K_0(\n);N)^{\epsilon} \to \mathcal{M}(K_0(\n);N)^{\epsilon}.
\end{align*}

%% file: seccoh-delta.tex
\subsection{Pullback to $\mathbb{G}_m$}\label{Delta}
Using the local maps studied in Section \ref{LocalHecke} we attach distribution valued cohomology classes to modular symbols. 

Firstly, given any finite set $S$ of finite places of $F$ we use (\ref{product}) to define a semi-local version
\begin{align*}
 \delta_S=\otimes_{\mathfrak{q} \in S} \delta_\mathfrak{q} \colon C_c(F_S,\Z) \to \St_S
\end{align*}
of the map $\delta_\mathfrak{q}$ of Section \ref{LocalHecke}.
As in the local case, this map is bijective and $B(F_{S})$-equivariant.

For a set $S$ of places as before, an open subgroup $\widetilde{U}\subseteq U^{\infty}$, a ring $R$ and an $R$-module $N$ we define
\begin{align*}
 \mathcal{D}(\widetilde{U},S;N)=\Dist(\I\sinfty/\widetilde{U}^S\times F_S,N).
\end{align*}
By (\ref{product2}) and (\ref{discrete}) we have an $\I^{\infty}$-equivariant isomorphism
\begin{align}\label{ident}
 \mathcal{D}(\widetilde{U},S;N)\cong C(\I\sinfty/\widetilde{U}^S, \Dist(F_S,N)).
\end{align}
We are mostly interested in the case where $\tild{U}=U(\m)$ for a non-zero ideal $\m$ of $\OO$.
To shorten the notation we set $\mathcal{D}(\m,S;N)=\mathcal{D}(U(\m),S;N)$ and write $\mathcal{D}(\mathfrak{m};N)$ instead of $\mathcal{D}(U(\m),\emptyset;N)$.

If $K\subseteq G(\A^{\infty})$ is a compact open subgroup, we write $U(K) \subseteq U^{\infty}$ for the preimage of $K$ under $\iota$. 
Next, we construct a global version
\begin{align*}
 \Delta_{g,S} = \Delta_{g}(K,S,N,\epsilon) \colon \Ah(K,S;N)(\epsilon) \too \mathcal{D}(U(gKg\inv),S;N)(\epsilon)
\end{align*}
of the dual of the semi-local map $\delta_{S}$ for every compact open subgroup $K\subseteq G(\A^{\infty})$ and every $g\in G(\A^{S,\infty})$.
Given $\Phi \in \Ah(K,S;N)$, $x \in \I\sinfty/U(gKg\inv)$ and $\fii \in C_c(F_S,\Z)$ we define
\begin{align*}
 \Delta_{g,S}(\Phi)(x,\fii)= \Phi \left(\begin{pmatrix} x & 0 \\ 0 & 1 \end{pmatrix} g\right) \left(\delta_{S}(\fii)\right)([0:1]-[1:0])
\end{align*}
using the identification (\ref{ident}).

We have an $\I$-action on $\mathcal{D}(U(K),S;N)(\epsilon)$ induced by left-multiplication on $\I\sinfty \times F_S$ and an $\I$-action on $\Ah(K,S;N)(\epsilon)$ via the embedding $\iota\colon \I \to G(\A)$.
As an immediate consequence of Lemma \ref{equivariance} we get

\begin{Proposition}\label{equivariance2}
$\Delta_{g,S}$ is $\I$-equivariant.
\end{Proposition}

Hence, the restriction map on group cohomology together with $\Delta_{g,S}$ induces a map of cohomology groups
\begin{align*}
 H^{i}(G(F),\Ah(K,S;N)(\epsilon)) &\too H^{i}(F^{\ast},\Ah(K,S;N)(\epsilon))\\
				  &\too H^{i}(F^{\ast},\mathcal{D}(U(g^{-1}Kg),S;N)(\epsilon))
\end{align*}
for every integer $i\geq 0$.
In particular, we get a map
\begin{align*}
 \mathcal{M}(K,S;N)^\epsilon\too H^{d-1}(F^{\ast},\mathcal{D}(U(g^{-1}Kg),S;N)(\epsilon)),
\end{align*}
which we will also denote by $\Delta_{g,S}$.

We want to discuss a special case of the above situation.
For a non-zero ideal $\m = \prod_\p \p^{m_\p}$ of $\OO$ we define an element $g_\m \in G(\A^\infty)$ by
\begin{align}\label{matrix}
 (g_\m)_\p = 	\begin{cases}
		 \begin{pmatrix} 1 & 0 \\ 0 & 1 \end{pmatrix} & \mbox{if $m_\p = 0$,} \\
                 \begin{pmatrix} \varpi_\p^{m_\p} & 1 \\ 0 & 1 \end{pmatrix} & \mbox{else,}
                \end{cases} 
\end{align}
where $\varpi_\p$ is a local uniformizer at $\p$.

If $\n = \prod_\p \p^{n_\p}$ is another non-zero ideal, we have $U(g_\m^S K_0(\n) (g^S_\m)\inv) = U(\m)^\infty$, where $(g_\m)^S$ is the projection of $g_\m$ on $G(\A\sinfty)$.
Therefore, we get maps
\begin{align*}
 \Delta_{\m,S}=\Delta_{(g_\m)^S,S}\colon \mathcal{M}(K_0(\n),S;N)^\epsilon \too H^{d-1}(F^\ast,\mathcal{D}(\m,S;N)(\epsilon))
\end{align*}
for all nonzero ideals $\m$ in $\OO$.
As always, we write $\Delta_{\mathfrak{m}}$ instead of $\Delta_{\mathfrak{m},\emptyset}$.
Let $$\mathcal{E} \colon H^{d-1}(F^{\ast},\mathcal{D}(\m,S,N)(\epsilon)) \too H^{d-1}(F^{\ast},\mathcal{D}(\m,N)(\epsilon))$$ be the map induced by the embedding
$$C_{c}^{0}(\I^{\infty}/U(\m)^{\infty},N)\too C_{c}^{0}(\I\sinfty/U(\m)\sinfty\times F_{S},N).$$
As in the local case, for a finite place $\mathfrak{p}\notin S$ the projection $\Pi_\m \colon \I/U(\m\p) \to \I/U(\m)$ yields maps
\begin{align*}
 (\Pi_\m)^\ast &\colon H^{d-1}(F^{\ast},\mathcal{D}(\m,S,N)(\epsilon)) \too H^{d-1}(F^{\ast},\mathcal{D}(\m\p,S,N)(\epsilon))
\intertext{and}
 (\Pi_\m)_\ast &\colon H^{d-1}(F^{\ast},\mathcal{D}(\m\p,S,N)(\epsilon)) \too H^{d-1}(F^{\ast},\mathcal{D}(\m,S,N)(\epsilon)).
\end{align*}
The following lemma is a direct consequence of the local Lemmas \ref{LocalHeckeLemma} and \ref{LocalDiagram}.

\begin{Lemma}\label{GlobalLemma} \renewcommand{\labelenumi}{(\roman{enumi})}
\begin{enumerate}
\item Let $\mathfrak{p}$ be a finite place of $F$, which is not contained in $S$.
Then, if $\p \nmid \m$, we have
\begin{align*}
 \Delta_{\m,S}(T_\p \kappa) &= \varpi_\p\inv. \Delta_{\m,S}(\kappa) + (\Pi_\m)_\ast(\Delta_{\m \p,S}(\kappa)) + \cf_\p(\n) \varpi_\p. \Delta_{\m,S}(\kappa)
\intertext{ and if $\mathfrak{p}\mid\m$, the equality}
 \Delta_{\m,S}(T_\p \kappa) &= (\Pi_\m)_\ast(\Delta_{\m \p,S}(\kappa))  + \cf_\p(\n) (\Pi_{\m\p\inv})^\ast(\Delta_{\m\p\inv,S}(\kappa))
\end{align*}
holds.
Here $\varpi_\p$ is an idele which is a local uniformizer at the $\p$-component and is equal to $1$ at all other places.

\item Assume that $\mathfrak{q}$ divides $\m$ for all $\mathfrak{q}\in S$.
Then the following diagram is commutative:
\begin{center}
 \begin{tikzpicture}
    \path 	(0,0) 	node[name=A]{$\mathcal{M}(K_0(\n),S;N)^\epsilon$}
		(6,0) 	node[name=B]{$\mathcal{M}(K_0(\n);N)^\epsilon$}
		(0,-2) 	node[name=C]{$H^{d-1}(F^{\ast},\mathcal{D}(\m,S:N)(\epsilon))$}
		(6,-2) 	node[name=D]{$H^{d-1}(F^{\ast},\mathcal{D}(\m;N)(\epsilon))$};
    \draw[->] (A) -- (B) node[midway, above]{$\Ev_{S}$};
    \draw[->] (A) -- (C) node[midway, left]{$\Delta_{\m,S}$};
    \draw[->] (B) -- (D) node[midway, right]{$\Delta_{\m}$};
    \draw[->] (C) -- (D) node[midway, below]{$\mathcal{E}$};
 \end{tikzpicture} 
\end{center}
(Remember that we assume that all $\mathfrak{q}\in S$ divide $\mathfrak{n}$ exactly once.)
\end{enumerate}
\end{Lemma}

%% file: seccoh-char.tex
\subsection{Homology classes attached to idele class characters}
We give a slight variation of the main result of Section 3 of \cite{DS} on the vanishing of certain homology classes attached to idele class characters of $F$.
It is used in the next section to give lower bounds on the order of vanishing of Stickelberger elements attached to modular symbols.

Given a non-zero ideal $\mathfrak{m}\subseteq\mathcal{O}$ and a ring $R$ we define
\begin{align}\label{funcspac}
 \mathcal{C}_{\natural}(\mathfrak{m},R)^{\sharp}=C_{\natural}^{0}(\I^{\sharp}/U(\mathfrak{m})^{\sharp},R)
\end{align}
for $\natural\in\left\{\emptyset, c\right\}$, $\sharp\in\left\{\emptyset, \infty\right\}$.
Let us fix a locally constant character $\chi\colon\I\to R^{\ast}$.
For a place $v$ of $F$ we denote by $\chi_{v}$ the local component of $\chi$ at $v$, i.e.~the composition $\chi_{v}\colon F^{\ast}_{v} \hookrightarrow \I  \stackrel{\chi}{\longrightarrow} R^{\ast}$.
Since the kernel of $\chi$ is open, there exists a non-zero ideal $\mathfrak{m}\subseteq\mathcal{O}$ such that the ramification of $\chi$ is bounded by $\mathfrak{m}$, i.e.~$\chi$ restricted to $U(\mathfrak{m})$ is trivial.
Hence, we have a product decomposition $\chi((x_{v})_{v})=\prod_{v}\chi_{v}(x_{v})$.
In the following we always assume that $\chi$ is trivial on principal ideles.
Therefore, we can view $\chi$ as a character $\chi\colon \I/F^{\ast} U(\mathfrak{m})\to R^{\ast}$ or alternatively as an element of $H^{0}(F^{\ast},\mathcal{C}(\mathfrak{m},R))$.
For every idele $x\in \I$ we have the identity of homology classes
\begin{align}\label{idelemult}
 x.\chi =\chi(x^{-1})\cdot\chi \in H^{0}(F^{\ast},\mathcal{C}(\mathfrak{m},R)).
\end{align}

Next, we use a version of Poincare duality to attach a homology class to $\chi$.
By Dirichlet's unit theorem the group of totally positive global units $E_{+}$ is free of rank $d-1$.
It follows that the homology group $H_{d-1}(E_{+},\Z)$ is free of rank $1$.
Let us fix once and for all a generator $\eta$ of this group.
For example, every ordering of the Archimedean places of $F$ gives a canonical choice for $\eta$ (see Remark 2.1 of \cite{Sp2} for details).
Let $\mathcal{F}$ be a fundamental domain for the action of $F^{\ast}/E^{+}$ on $\I/U$.
The isomorphism
\begin{align*}
 \mathcal{C}_{c}(\mathcal{O}, \Z)=C_{c}(\I/U,\Z)= \cind_{U}^{\I}\Z \cong \cind_{E_{+}}^{F^{\ast}}C(\mathcal{F},\Z)
\end{align*}
of $F^{\ast}$-modules together with Shapiro's Lemma imply that
\begin{align}\label{class}
 H_{d-1}(E_{+},C(\mathcal{F},\Z))\cong H_{d-1}(F^{\ast}, \mathcal{C}_{c}(\mathcal{O},\Z)).
\end{align}
We define $\vartheta$ as the image of the cap-product $\cf_{\mathcal{F}}\cap \eta\in H_{d-1}(E_{+},C(\mathcal{F},\Z))$ under (\ref{class}).
The class $\vartheta$ is in fact $\I$-invariant.
The pairing
\begin{align*}
 \mathcal{C}(\mathfrak{m},R)\times\mathcal{C}_{c}(\mathcal{O},\Z)\too \mathcal{C}_{c}(\mathfrak{m},R),\ (\phi,\psi) \mapstoo \phi\cdot \psi
\end{align*}
induces a cap-product pairing
\begin{align*}
 \cap \colon H^{i}(F^{\ast},\mathcal{C}(\mathfrak{m},R))\times H_{j}(F^{\ast},\mathcal{C}_{c}(\mathcal{O}, \Z))\too H_{j-i}(F^{\ast},\mathcal{C}_{c}(\mathfrak{m},R))
\end{align*}
for $j\geq i\geq 0$.
As a special case, by taking the cap-product with $\vartheta$ we get a homomorphism of $R$-modules
\begin{align}\label{cap}
 H^{0}(F^{\ast},\mathcal{C}(\mathfrak{m},R))\stackrel{\cdot\ \cap\vartheta}{\longrightarrow} H_{d-1}(F^{\ast},\mathcal{C}_{c}(\mathfrak{m},R)).
\end{align}
It is $\I$-equivariant since $\vartheta$ is $\I$-invariant.

As in Section \ref{cohomology}, we fix locally constant characters $\epsilon_{v}\colon F_{v}^\ast \to \left\{\pm 1\right\}$ for all $v\in S_{\infty}$ and put $\epsilon=\prod_{v\in S_{\infty}}\epsilon_{v}$.
Tensoring the identity map of $\mathcal{C}_c(\m,R)^\infty$ with the $F_{\infty}^{\ast}$-equivariant homomorphism
\begin{align*}
 C(F_{\infty}^{\ast}/U_{\infty},\Z)\too \Z(\epsilon),\ 
 f\mapstoo \sum_{x\in F_{\infty}^{\ast}/U_{\infty}} \epsilon(x)f(x)
\end{align*}
yields an $\I$-equivariant $R$-linear map
\begin{align}\label{epsilon}
 \mathcal{C}_{c}(\mathfrak{m},R)\cong \mathcal{C}_{c}(\mathfrak{m},R)^{\infty} \otimes C(F_{\infty}^{\ast}/U_{\infty},\Z)\too \mathcal{C}_{c}(\mathfrak{m},R)^{\infty}(\epsilon).
\end{align}

Given a finite set $S$ of finite places of $F$ we define the following generalization of (\ref{funcspac}):
\begin{align*}
 \mathcal{C}_{c}(\mathfrak{m},S,R)^{\infty}=C_{c}(F_S \times I^{S,\infty}/U(\mathfrak{m})^{S,\infty},R)
\end{align*}
Extension by zero induces an $\I$-equivariant $R$-linear map
\begin{align}\label{extension}
 \mathcal{C}_{c}(\mathfrak{m},R)^{\infty}(\epsilon)\too \mathcal{C}_{c}(\mathfrak{m},S,R)^{\infty}(\epsilon).
\end{align}
Let $c_{\chi}=c_{\chi}(\mathfrak{m},S,\epsilon)$ denote the image of $\chi$ under the composition
\begin{align}\label{homclass}
\begin{split}
 H^{0}(F^{\ast},\mathcal{C}(\mathfrak{m},R)) &\stackrel{(\ref{cap})_{\color{white}\ast\color{black}}}{\too} H_{d-1}(F^{\ast},\mathcal{C}_{c}(\mathfrak{m},R))\\
					     &\stackrel{(\ref{epsilon})_{\ast}}{\too} H_{d-1}(F^{\ast}, \mathcal{C}_{c}(\mathfrak{m},R)^{\infty}(\epsilon))\\
					     &\stackrel{(\ref{extension})_{\ast}}{\too} H_{d-1}(F^{\ast}, \mathcal{C}_{c}(\mathfrak{m},S,R)^{\infty}(\epsilon)).
\end{split}
\end{align}
Equation (\ref{idelemult}) together with the fact that (\ref{homclass}) is $\I$-equivariant immediately implies that
\begin{align}\label{idelemult2}
 x.c_{\chi}=\chi(x^{-1})\cdot c_{\chi}
\end{align}
for every idele $x\in\I$.
The following proposition is essentially Proposition 3.8 of \cite{DS}.

\begin{Proposition}\label{vanish1}
Assume we have given ideals $\mathfrak{a}_{v}\subseteq R$ for $v\in S\cup S_{\infty}$ such that
\begin{enumerate}[(A)]
\item\label{A} for every $\q \in S$ we have $\chi_{\q}(x_{\q}) \equiv 1 \bmod \mathfrak{a}_{\q}$ for all $x_{\q}\in F_{\q}^{\ast}$,
\item for every $v\in S_{\infty}$ we have $\chi_{v}(-1)=-\epsilon_{v}(-1) \bmod \mathfrak{a}_{v}$,
\item $\prod_{v\in S\cup S_{\infty}}\mathfrak{a}_{v}=0$.
\end{enumerate}
Then the homology class $c_{\chi}$ vanishes.
\end{Proposition}

We finish this section with a few words on the proof of the proposition. To simplify matters we neglect the contribution from the Archimedean places. Let us fix ideals $\mathfrak{a}_\q\subseteq R$, $\q\in S$, such that property \eqref{A} holds. In addition, let $\mathfrak{a}\subseteq R$ be an ideal, which contains all the ideals $\mathfrak{a}_\q$. Thus, the character
\begin{align*}
\overline{\chi}\colon \I^{S} \stackrel{\chi^{S}}{\too} R^{\ast}\too \left(R/\mathfrak{a}\right)^{\ast}
\end{align*}
defines an element in $H^{0}(F^{\ast},\mathcal{C}(\mathfrak{m},R/\mathfrak{a})^{S})$, where
$$\mathcal{C}(\mathfrak{m},R/\mathfrak{a})^{S}=C^{0}(\I^{S}/U(\mathfrak{m})^{S},R/\mathfrak{a}).$$
Similarly as above, we define a fundamental class $\vartheta^{S}$ for the group of $S$-units in $F^{\ast}$ and put $\overline{c_\chi}=\overline{\chi}\cap \vartheta^{S}\in H_{d-1+\left|S\right|}(F^{\ast},\mathcal{C}_{c}(\mathfrak{m},R/\mathfrak{a})^{S})$.

By our assumptions, the maps
$$\dd\chi_\q\colon F^{\ast}_\q\too \mathfrak{a}_q/\mathfrak{a}\mathfrak{a}_q,\quad x\mapstoo \chi_{q}(x)-1 \bmod \mathfrak{a}\mathfrak{a}_q$$
are in fact homomorphisms for all $\q\in S$. Thus,
$$z_{\dd\chi_\q}(x)(y)= \cf_{x\OO_\q}(y)\cdot \dd\chi_\q(x)+ ((\cf_{\OO_\q}-\cf_{x\OO_\q})\cdot\dd\chi_\q)(y)$$
defines a $1$-cocyle on $F_{\p}^{\ast}$ with values in $C_{c}^{0}(F_\q,\mathfrak{a}_\q/\mathfrak{a}\mathfrak{a}_\q)$. We write $c_{\dd\chi_\q}$ for the associated cohomology class.

If we assume that that $\mathfrak{a}\cdot\prod_{\q\in s}\mathfrak{a}_\q=0$ holds, we have a well-defined multiplication map
$$R/\mathfrak{a}\times \prod_{\q\in S}\mathfrak{a}_\q/\mathfrak{a}\mathfrak{a}_q\too \prod_{\q\in S}\mathfrak{a}_q,$$
which induces a homomorphism
$$\bigotimes_{\q\in S}C_{c}^{0}(F_\q,\mathfrak{a}_\q/\mathfrak{a}\mathfrak{a}_\q) \otimes \mathcal{C}_{c}(\mathfrak{m},R/\mathfrak{a})^{S}\too \mathcal{C}_{c}(\mathfrak{m},S,\prod_{\q\in S}\mathfrak{a}_\q).$$
A mostly formal computation in cohomology then yields the equality
$$c_\chi=\pm (c_{\dd\chi_{\q_1}}\cup\ldots\cup c_{\dd\chi_{\q_r}})\cap \overline{c_\chi}$$
with $S=\{\q_1,\ldots,\q_r\}$.
Specialising to the case $\mathfrak{a}=R$ we obtain the claim.

%% file: seccoh-vanish.tex
\subsection{Stickelberger elements attached to modular symbols}\label{Stick}
We define Stickelberger elements attached to modular symbols and bound their order of vanishing from below. 

Throughout this section we fix a ring $R$, an $R$-module $N$, a non-zero ideal $\mathfrak{n}\subseteq\mathcal{O}$, a character $\epsilon$ as in Section \ref{cohomology} and a modular symbol $\kappa\in \mathcal{M}(K_{0}(\mathfrak{n}),N)^{\epsilon}$.
Further, we fix a finite abelian extension $L/F$ with Galois group $\mathcal{G}=\mathcal{G}_{L/F}$ and denote by $\rec=\rec_{L/F} \colon \I \to \mathcal{G} \subseteq \Z[\mathcal{G}]^{\ast}$ the Artin reciprocity map.
For a finite place $\mathfrak{p}$ of $F$, which is unramified in $L$, we let $\sigma_{\mathfrak{p}}\in\mathcal{G}$ be the (arithmetic) Frobenius element at $\p$.
More generally, if $\mathfrak{a}=\prod_{\p}\p^{a_\p}\subseteq \mathcal{O}$ is a (non-empty) product of powers of prime ideals, which are unramified in $L$, we put $\sigma_{\mathfrak{a}}=\prod_{\p}\sigma_{\p}^{a_\p}$.

Let $\mathfrak{m}\subseteq\mathcal{O}$ be a non-zero ideal such that the ramification of $L/F$ is bounded by $\m$, i.e.~the Artin reciprocity map is trivial on $U(\mathfrak{m})$. 
Then we can identify the reciprocity map with a cohomology class in $H^{0}(F^{\ast},\mathcal{C}(\mathfrak{m},\Z[\mathcal{G}]))$.
Let $\rho_{L/F}=\rho_{L/F,\mathfrak{m}}\in H_{d-1}(F^{\ast},\mathcal{C}_c(\mathfrak{m},\Z[\mathcal{G}])^\infty(\epsilon))$ be the image of this cohomology class under (\ref{homclass}) with $S=\emptyset$.
The natural pairing
\begin{align*}
 \mathcal{C}_{c}(\mathfrak{m},\Z[\G])^{\infty}\times \mathcal{D}(\mathfrak{m};N)\too \Z[\G]\otimes N
\end{align*}
induces a cap-product pairing
\begin{align*}
 H_{d-1}(F^{\ast},\mathcal{C}_{c}(\mathfrak{m},\Z[\G])^{\infty}(\epsilon))\times H^{d-1}(F^{\ast},\mathcal{D}(\mathfrak{m};N)(\epsilon))\too \Z[\G]\otimes N.
\end{align*}
By construction we have
\begin{align}\label{equivpair}
 c \cap x.b = x\inv. c \cap b
\end{align}
for all $x\in\I$, $c \in H_{d-1}(F^{\ast},\mathcal{C}_{c}(\mathfrak{m},\Z[\G])^{\infty}(\epsilon))$ and $b \in H^{d-1}(F^{\ast},\mathcal{D}(\mathfrak{m};N)(\epsilon))$.

\begin{Def}\label{Stickelberger}
The Stickelberger element of modulus $\mathfrak{m}$ associated to $\kappa$ and $L/K$ is defined as the cap-product
\begin{align*}
\Theta_{\mathfrak{m}}(L/F,\kappa)=\rho_{L/F} \cap \Delta_{\mathfrak{m}}(\kappa) \in \Z[\G]\otimes N.
\end{align*}
\end{Def}

Let $k$ be a field, which is an $R$-algebra.
A character $\chi\colon \mathcal{G}\to k^{\ast}$ induces a homomorphism $R[\mathcal{G}]\to k$ of $R$-algebras and hence a homomorphism of $R$-modules $$\chi\colon \Z[\G]\otimes N \too k\otimes_{R}N.$$
We can also consider $\chi$ as a character on $\I$ via the reciprocity map.
We write $\chi_{\infty}\colon F_{\infty}^{\ast}\to \mu_{2}(k)$ for the component at infinity of $\chi$.
The Stickelberger elements satisfy the following compatibility relations:

\begin{Proposition}\label{compatibility} \renewcommand{\labelenumi}{(\roman{enumi})}
\begin{enumerate}
\item Let $L^\prime$ be an intermediate extension of $L/F$. Then we have
\begin{align*}
 \pi_{L/L^\prime}(\Theta_{\mathfrak{m}}(L/F,\kappa))=\Theta_{\mathfrak{m}}(L^\prime/F,\kappa),
\end{align*}
where
\begin{align*}
 \pi_{L/L^{\prime}}\colon \Z[\mathcal{G}_{L/F}]\otimes N\too \Z[\mathcal{G}_{L^\prime/F}]\otimes N
\end{align*}
is the canonical projection.
\item Let $\mathfrak{p}$ be a finite place of $F$, which does not divide $\mathfrak{m}$.
Assume that $\kappa$ is an eigenvector of $T_\mathfrak{p}$ with eigenvalue $\lambda_{\mathfrak{p}}$.
Then the following equality holds:
\begin{align*}
 \Theta_{\mathfrak{m}\mathfrak{p}}(L/F,\kappa)
 =(\lambda_{\mathfrak{p}}-\sigma_{\mathfrak{p}}\inv-\cf_{\p}(\n)\sigma_{\mathfrak{p}})\Theta_{\mathfrak{m}}(L/F,\kappa)
\end{align*}
\item Let $\mathfrak{p}$ be a finite place of $F$, which does divide $\mathfrak{m}$ and write $r=\ord_\p(\m)$.
Assume that $\kappa$ is an eigenvector of $T_\mathfrak{p}$ with eigenvalue $\lambda_{\mathfrak{p}}$.
Then we have a decomposition
\begin{align*}
 \Theta_{\mathfrak{m}\mathfrak{p}}(L/F,\kappa) =\lambda_{\mathfrak{p}}\Theta_{\mathfrak{m}}(L/F,\kappa)\ +\ \cf_{\p}(\n) v_{\mathfrak{m}}(\Theta_{\m\p\inv}(L/F,\kappa)),
\end{align*}
where the elements $v_{\mathfrak{m}}(\Theta_{\m\p\inv}(L/F,\kappa))$ satisfy the following properties:
\begin{itemize}
 \item $\pi_{L/L^{\prime}}(v_{\mathfrak{m}}(\Theta_{\m\p\inv}(L/F,\kappa)))=v_{\mathfrak{m}}(\Theta_{\m\p\inv}(L^{\prime}/F,\kappa))$ for all intermediate extensions $L^{\prime}$ of $L/F$
 \item $v_{\mathfrak{m}}(\Theta_{\m\p\inv}(L/F,\kappa))= [U_{\p}^{(r-1)}:U_{\p}^{(r)}](\Theta_{\m\p\inv}(L/F,\kappa))$ if the Artin reciprocity map for $L/F$ is trivial on $U(\m\p\inv)$
 \item Let $k$ be a field, which is an $R$-algebra, and $\chi\colon \mathcal{G}\to k^{\ast}$ a character such that $\chi_{\p}$ has conductor $\mathfrak{p}^{r}$. Then:
\begin{align*}
 \chi(v_{\mathfrak{m}}(\Theta_{\m\p\inv}(L/F,\kappa)))=0
\end{align*}
\end{itemize}
\item Suppose $k$ is a field, which is an $R$-algebra, and $\chi\colon \mathcal{G}\to k^{\ast}$ is a character. If $\chr(k)\neq 2$ and $\chi_{\infty}\neq \epsilon$, we have
\begin{align*}
 \chi(\Theta_{\mathfrak{m}}(L/F,\kappa))=0.
\end{align*}
\end{enumerate}
\end{Proposition}

\begin{proof}
(i) By global class field theory the diagram
\begin{center}
  \begin{tikzpicture}
    \path 	(0,0) 	node[name=A]{$\I/U(\mathfrak{m})$}
		(3,0) 	node[name=B]{$\mathcal{G}_{L/F}$}
		(5,0)   node[name=E]{$\Z[\mathcal{G}_{L/F}]$}
		(0,-2) 	node[name=C]{$\I/U(\mathfrak{m})$}
		(3,-2) 	node[name=D]{$\mathcal{G}_{L^{\prime}/F}$}
		(5,-2)  node[name=F]{$\Z[\mathcal{G}_{L^{\prime}/F}]$};
    \draw[->] (A) -- (B) node[midway, above]{$\rec_{L/F}$};
    \draw[->] (A) -- (C) node[midway, left]{=};
    \draw[->] (B) -- (D) node[midway, right]{$\pi_{L/L^{\prime}}$};
    \draw[->] (C) -- (D) node[midway, below]{$\rec_{L{^\prime}/F}$};
		\draw[->] (E) -- (F) node[midway, right]{$\pi_{L/L^{\prime}}$};
		\draw[right hook->] (B) -- (E);
		\draw[right hook->] (D) -- (F);
  \end{tikzpicture} 
  \end{center}
is commutative and thus, the claim follows.\\
(ii) By assumption and Lemma \ref{GlobalLemma} we have
\begin{align*}
 \lambda_{p}\Theta_{\mathfrak{m}}(L/F,\kappa) &=\rho_{L/F,\mathfrak{m}} \cap \Delta_{\mathfrak{m}}(T_{\mathfrak{p}}\kappa) \\
					      &=\rho_{L/F,\mathfrak{m}} \cap (\varpi_{\p}\inv.\Delta_{\mathfrak{m}}(\kappa) + \cf_{\p}(\n)\varpi_{\p}.\Delta_{\mathfrak{m}}(\kappa) + (\Pi_\m)_\ast(\Delta_{\m\p}(\kappa))).
\end{align*}
Equations \eqref{idelemult2} and \eqref{equivpair} imply
\begin{align*}
\rho_{L/F,\mathfrak{m}}\cap \varpi_{\p}\inv.\Delta_{\mathfrak{m}}(\kappa) &= \varpi_{\p}.\rho_{L/F,\mathfrak{m}}\cap \Delta_{\mathfrak{m}}(\kappa) \\
									  &= \rec(\varpi_{\p}\inv) \rho_{L/F,\mathfrak{m}}\cap \Delta_{\mathfrak{m}}(\kappa) \\
									  &= \sigma_{\p}^{-1} \Theta_{\mathfrak{m}}(L/F,\kappa)
\end{align*}
for the first summand. The second summand is dealt with similarly.
Finally, for the third summand the following equality holds:
\begin{align*}
\rho_{L/F,\mathfrak{m}}\cap (\Pi_\m)_\ast(\Delta_{\m\p}(\kappa)) &= (\Pi_\m)^\ast \rho_{L/F,\mathfrak{m}}\cap (\Delta_{\m\p}(\kappa)) \\
								 &=\rho_{L/F,\m\p} \cap (\Delta_{\m\p}(\kappa)) \\
								 &=\Theta_{\mathfrak{m}\mathfrak{p}}(L/F,\kappa)
\end{align*}\\
(iii) As before, Lemma \ref{GlobalLemma} implies
\begin{align*}
 \lambda_{p}\Theta_{\mathfrak{m}}(L/F,\kappa)&=\rho_{L/F,\mathfrak{m}} ((\Pi_\m)_\ast(\Delta_{\m \p,S}(\kappa))  + \cf_\p(\n) (\Pi_{\m\p\inv})^\ast(\Delta_{\m\p\inv,S}(\kappa))).
\end{align*}
The first summand was already dealt with in part (ii). For the second summand we get
\begin{align*}
 \rho_{L/F,\mathfrak{m}}\cap(\Pi_{\m\p\inv})^\ast(\Delta_{\m\p\inv,S}(\kappa))= (\Pi_{\m\p\inv})_\ast(\rho_{L/F,\mathfrak{m}})\cap \Delta_{\m\p\inv,S}(\kappa).
\end{align*}
The listed properties of $v_{\mathfrak{m}}(\Theta_{\m\p\inv}(L^{\prime}/F,\kappa))$ now follow as in (i), respectively by using orthogonality of characters.\\
(iv) We have
\begin{align*}
 \chi(\Theta_{\mathfrak{m}}(L/F,\kappa)) &=\chi(\rho_{L/F})\cap \Delta_{\m}(\kappa) \\
					 &=c_{\chi}(\m,\epsilon)\cap \Delta_{\m}(\kappa).
\end{align*}
Using orthogonality of characters one sees that $c_{\chi}(\m,\epsilon)$ vanishes unless $\chi=\epsilon$.
\end{proof}

\begin{Remark}\label{padic}\renewcommand{\labelenumi}{(\roman{enumi})}
\begin{enumerate}
\item Let us assume that $R$ is a subring of $\C$ and $N\subseteq \C$ is an $R$-submodule.
Since the set of characters is a basis of the dual of $\C[\mathcal{G}]$, it follows that an element $\Theta\in \Z[\mathcal{G}]\otimes N$ is uniquely determined by the values $\chi(\Theta)$ for $\chi$ ranging through all complex-valued characters of $\mathcal{G}$.
Now assume that $\kappa\in \mathcal{M}(K_{0}(\mathfrak{n}),N)^{\epsilon}$ is an eigenvector for all Hecke operators $T_{\mathfrak{p}}$.
Then by the previous proposition the Stickelberger elements $\Theta_{\mathfrak{m}}(L/F,\kappa)$ (for varying $\mathfrak{m}$ and $L$) are uniquely determined by the $T_{\mathfrak{p}}$-eigenvalues and the complex numbers $\chi(\Theta_{\mathfrak{m}}(L/F,\kappa))$, where $\chi$ ranges through all primitive characters of conductor $\mathfrak{m}$ such that $\chi_\infty=\epsilon$.
\item Using the above Proposition one can construct $p$-adic $L$-functions of ``ordinary'' modular symbols: Let $p$ be a rational prime and let $R$ be the valuation ring of a $p$-adic field.
Let $\n\subset \mathcal{O}$ be a non-zero ideal such that every place $\p$ lying over $p$ divides $\n$.
Further, let $\kappa\in \mathcal{M}(K_{0}(\mathfrak{n}),R)^{\epsilon}$ be an eigenvector of $T_\p$ for all places $\p$ of $F$ dividing $p$ with eigenvalue $\lambda_{p}\in R^{\ast}$.
For every ideal $\m\subset \mathcal{O}$ of the form $\m=\prod_{\p\mid p}\p^{m_\p}$ with $m_\p\geq 1$ for all $\p$ we define 
$$\widetilde{\Theta}_{\m}(\cdot/F,\kappa)=\left(\prod_{\p \mid p} \lambda_{p}^{-m_\p}\right) {\Theta}_{\m}(\cdot/F,\kappa).$$
Then by Proposition \ref{compatibility} (iii) this is a norm-compatible family and hence defines an element $\widetilde{\Theta}(F,\kappa)$ in the completed group ring $R[[\Gamma]]$, where $\Gamma$ is the Galois group of the maximal abelian extension of $F$ unramified outside $p$ and $\infty$.
For $s\in \Z_{p}$ and $\gamma \in \Gamma$ we put $\left\langle \gamma\right\rangle^{s}=\exp_{p}(s\ \log_{p}(\chi_{\cyc})(\gamma))$, where $\chi_{\cyc}\colon \Gamma\to \Z_{p}^{\ast}$ is the cyclotomic character.
The $p$-adic $L$-function of $\kappa$ is the locally analytic function 
$$L_{p}(s,\kappa)=\left\langle \cdot\right\rangle^{s}\left(\widetilde{\Theta}(F,\kappa)\right).$$ 
\end{enumerate}
\end{Remark}
 
Let $S$ be a finite set of finite places of $F$, such that $\q$ divides $\mathfrak{n}$ exactly once for all $\q \in S$.
We write $S_{\mathfrak{m}} = \left\{ \q \in S \mbox{ s.t. } \q \mid \mathfrak{m} \right\}$.
For a place $v$ of $F$ we let $\mathcal{G}_{v}\subseteq \mathcal{G}$ be the decomposition group at $v$.
If $\q \in S$, we define $I_{\q} \subseteq \Z[\mathcal{G}]$ as the kernel of the projection $\Z[\mathcal{G}]\onto \Z[\mathcal{G}/\mathcal{G}_{\q}]$.
If $v\in S_{\infty}$, we let $\sigma_{v}$ be a generator of $\mathcal{G}_{v}$ and define $I_{v}^{\pm 1}\subseteq \Z[\mathcal{G}]$ as the ideal generated by $\sigma_{v}\mp 1$.

\begin{Proposition}\label{vanish2}
Assume that there exists an $S$-special modular symbol $\kappa'\in \mathcal{M}(K_{0}(\mathfrak{n}),S,N)^\epsilon$ which lifts $\kappa$, i.e. $\Ev_S(\kappa')=\kappa$ and that $N$ is $\Z$-flat.
Then:
\begin{align*}
 \Theta_{\mathfrak{m}}(L/F,\kappa)\in \left(\prod_{v\in S_{\infty}} I_{v}^{-\epsilon_{v}(-1)} \cdot \prod_{\q \in S_\mathfrak{m}} I_{\q}\right) \otimes N
\end{align*}
In particular, if $L/F$ is totally real, $\epsilon$ is the trivial character and $N=R$, we have
\begin{align*}
 2^{-d}\Theta_{\mathfrak{m}}(L/F,\kappa)&\in R[\mathcal{G}] \intertext{and} \ord_{R}(2^{-d}\Theta_{\mathfrak{m}}(L/F,\kappa))&\geq \left|S_{\m}\right|
\end{align*}
with $d=[F:\Q]$.
\end{Proposition}

\begin{proof}
Let $\kappa'$ be an $S$-special lift of $\kappa$.
Evaluating $\kappa'$ at $\otimes_{\p \in S\setminus S_{\mathfrak{m}}}\fii_{\p}$ gives an $S_{\mathfrak{m}}$-special modular symbol, which lifts $\kappa$.
Hence, we may assume that $S=S_{\mathfrak{m}}$.
Let us write $\rho_{L/F}(S)\in H_{d-1}(F^\ast,\mathcal{C}_c(\mathfrak{m},S,\Z[\mathcal{G}])^\infty(\epsilon))$ for the image of $\rho_{L/F}$ under (\ref{extension}).
The equality
\begin{align*}
 \Theta_{\mathfrak{m}}(L/F,\kappa)=\rho_{L/F}\cap \Delta_{\mathfrak{m}}(\kappa)=\rho_{L/F}(S)\cap\Delta_{\mathfrak{m},S}(\kappa')
\end{align*}
follows by Lemma \ref{GlobalLemma} (ii).
We put $A=\quot{\Z[\G]}{\prod_{v\in S_{\infty}} I_{v}^{-\epsilon_{v}(-1)} \cdot \prod_{\q \in S} I_{\q}}$ and let $\pi\colon \Z[\G]\to A$ be the projection map (resp.~the projection map $\pi\colon \Z[\G]\otimes N \to A\otimes N$).
Then we have
\begin{align*}
 \pi (\Theta_{\mathfrak{m}}(L/F,\kappa))=\pi_{\ast}(\rho_{L/F}(S))\cap \Delta_{\mathfrak{m},S}(\kappa')=0
\end{align*}
since by Proposition \ref{vanish1} the homology class $\pi_{\ast}(\rho_{L/F}(S))=c_{\pi\circ \rec}(\mathfrak{m},S,\epsilon)$ vanishes.
\end{proof}

There is a natural involution on the group algebra
\begin{align*}
 \Z[\mathcal{G}]\too \Z[\mathcal{G}],\ \Theta\mapstoo \Theta^{\vee},
\end{align*}
which is induced by the the inversion map
\begin{align*}
 \mathcal{G}\too\mathcal{G},\ g\mapstoo g^{-1}.
\end{align*}
In certain situations, Stickelberger elements satisfy functional equations with respect to this involution.

\begin{Proposition}\label{funceq} \renewcommand{\labelenumi}{(\roman{enumi})}
Assume that $\n$ can be decomposed as a product $\n_1\cdot \n_2$ of coprime ideals such that $\n_1$ and $\m$ are coprime and $\n_{2}\mid\m$.
Write $\n_{1}=\prod_{i=1}^{r}\p_{i}^{n_i}$ as a product over powers of prime ideals with $n_{i}\geq 1$ for all $1\leq i\leq r$.
If $\kappa$ is an eigenvector of the Atkin-Lehner involutions $W_{\p_i^{n_i}}$, $1\leq i \leq r$, with eigenvalues $\varepsilon_{\p_i}\in \left\{\pm 1\right\}$, we have
$$\Theta_{\mathfrak{m}}(L/F,\kappa)^\vee=(-1)^{d}\epsilon(-1)\cdot\varepsilon_{\n_1}\cdot\sigma_{\n_1}^{-1}\cdot\Theta_{\mathfrak{m}}(L/F,\kappa),$$
where $\varepsilon_{n_{1}}=\prod_{i=1}^{r}\varepsilon_{\p_i}$ is the eigenvalue of the Atkin-Lehner involution $W_{\n_{1}}$. 
\end{Proposition}

\begin{proof}
By Lemma \ref{localfuneq} we have
$$\Delta_{\m}(W_{\n_{1}}\Phi)=(-1)\epsilon(-1)\varpi_{n_{1}}.\invol \left(\Delta_{\m}\left(
\begin{pmatrix} 0& 1\\
1& 0
\end{pmatrix}
\Phi\right)
\right)$$
for all $\Phi \in \Ah(K_{0}(\n);N)(\epsilon)$, where $\varpi_{n_{1}}$ is an idele which is the $n_{i}$-th power of a uniformizer at $\p_{i}$, $1\leq i \leq r$, and $1$ at every other place.
Hence, we get
\begin{align*}
 \varepsilon_{\n_1} \Theta_{\mathfrak{m}}(L/F,\kappa) &=  \Theta_{\mathfrak{m}}(L/F,W_{\n_1}\kappa) \\
						      &= (-1)\epsilon(-1) \rho_{L/F} \cap \varpi_{n_{1}}.\invol(\Delta_{\m}(\kappa)) \\
						      &= (-1)\epsilon(-1) \invol(\varpi_{\n_1}^{-1}.\rho_{L/F})\cap \Delta_{\m}(\kappa) \\
						      &= (-1)\epsilon(-1) \sigma_{n_1} \invol (\rho_{L/F}) \cap \Delta_{\m}(\kappa).
\end{align*}
Note that by applying the involution $\invol$ we invert the $F^{\ast}$-action on $\mathcal{C}_c(\mathfrak{m},\Z[\mathcal{G}])^\infty(\epsilon)$.
Inverting the $E$-action induces multiplication by $(-1)^{d-1}$ on the homology group $H_{d-1}(E,\Z)$ and thus, we get $$ \invol (\rho_{L/F})= (-1)^{d-1} (\rho_{L/F})^\vee $$ and the claim follows.
\end{proof}

As a consequence, we can compute the parity of the order of vanishing of Stickelberger elements.
With the same hypothesis as in the previous proposition we get:
\begin{Corollary}\label{parity}
Suppose that $N=R$, $\ord_{R}(\Theta_{\mathfrak{m}}(L/F,\kappa))=r<\infty$ and that $2$ acts invertibly on $I_{R}^{r}/I_{R}^{r+1}$. Then:
\begin{align*}
 (-1)^{r}=(-1)^{d}\epsilon(-1)\cdot\epsilon_{\n_1}
\end{align*}
\end{Corollary}

\begin{proof}
The involution $\left(\cdot\right)^\vee$ induces multiplication by $(-1)$ on $I_{R}/I_{R}^{2}$ and thus, it induces multiplication by $(-1)^{r}$ on $I_{R}^{r}/I_{R}^{r+1}$.
Since $\sigma_{\n_1}$ is congruent to $1$ modulo $I_{R}$, the equality
\begin{align*}
 (-1)^{r}\Theta_{\mathfrak{m}}(L/F,\kappa)=\Theta_{\mathfrak{m}}(L/F,\kappa)^\vee=(-1)^{d}\epsilon(-1)\cdot\epsilon_{\n_1}\cdot\Theta_{\mathfrak{m}}(L/F,\kappa)
\end{align*}
holds in $I_{R}^{r}/I_{R}^{r+1}$. Assume that $(-1)^{r}\neq(-1)^{d}\epsilon(-1)\cdot\epsilon_{\n_1}$, then the above equation would imply that
\begin{align*}
 2 \Theta_{\mathfrak{m}}(L/F,\kappa) \equiv 0 \bmod I^{r+1}_{R}.
\end{align*}
Since by assumption $2$ acts invertibly on $I_{R}^{r}/I_{R}^{r+1}$, we get that $\Theta_{\mathfrak{m}}(L/F,\kappa) \equiv 0 \bmod I^{r+1}_{R}$.
But this contradicts our assumption that the order of vanishing of $\Theta_{\mathfrak{m}}(L/F,\kappa)$ is exactly $r$.
\end{proof}

%% file: secaut-autom.tex
\subsection{Adelic Hilbert modular forms of parallel weight 2}\label{automorphic}
In this section we give a reminder on the definition of Hilbert modular forms of parallel weight $(2,\ldots,2)$ and the Eichler-Shimura map in the adelic setting.
Although our treatment is similar to the one of Spie\ss\ in \cite{Sp}, our construction of the Eichler-Shimura homomorphism is slightly different from the one given in loc.~cit.
Recall that the group $\PGL_2(\R)^{+}$ operates on the complex upper half plane $\HH=\{ z \in \C \mid \Ima(z)>0 \}$ via fractional linear transformations.
Hence, after fixing once and for all an ordering of the Archimedean places of $F$ we get an action of $G(F_\infty)^+=\prod_{v \in S_\infty} \PGL_2(F_v)^+$ on the symmetric space $X=\HH^{d}$, where $d=[F:\Q]$.
Let $K_\infty \subseteq G(F_\infty)^{+}$ be the image of $\SO(2)$ under the projection of $\GL_2(F_\infty)$ onto $G(F_\infty)$.
Thus, $K_\infty$ is the stabilizer of $(i,\ldots,i)\in X$ inside $G(F_\infty)^{+}$.
We consider the factor of automorphy $j(g,\underline{z})=\prod_{v \in S_{\infty}} j(g_v,z_v)$ for $g=(g_v)_{v \in S_\infty} \in G(F_\infty)^+$ and $\underline{z}=(z_v)_{v \in S_\infty} \in X$, where $j(\gamma,z) = \det(\gamma)^{-\frac{1}{2}} (cz+d)$ for $\gamma = \begin{pmatrix} a & b \\ c & d \end{pmatrix} \in G(\R)^+$ and $z \in \HH$.

\begin{Def}\label{forms}
The space $\mathcal{A}_0(G,\hol,\underline{2})$ of (adelic) Hilbert modular cusp forms of parallel weight $(2,\dots,2)$ is the space of complex valued functions $\Phi$ on $G(\A)$ such that:
\begin{enumerate} \renewcommand{\labelenumi}{(\roman{enumi})}
\item $\Phi(\gamma g) = \Phi(g)$ for all $\gamma \in G(F)$, $g \in G(\A)$.
\item $\Phi(g k_\infty) = j(k_\infty, \underline{i})^{-2} \Phi(g)$ for all $g \in G(\A)$, $k_\infty \in K_\infty^+$.
\item For $g\in G(\A^{\infty})$ and $\underline{z} \in X$ define $f_\Phi(\underline{z},g) := j(g_\infty, \underline{i})^2 \Phi(g_\infty,g)$, where $g_\infty \in G(F_\infty)^+$ is chosen such that $g_\infty \underline{i} = \underline{z}$. 
\item There is a compact open subgroup $K\subset G(\A^{\infty})$ such that $\Phi(gk)=\Phi(g)$ for all $g \in G(\A)$, $k \in K$.	The map $\underline{z} \mapstoo f_\Phi(\underline{z},g)$ is a holomorphic function on $X$ for all $g \in G(\A^{\infty})$.
\item (Cuspidality) We have $$ \int_{\A/F} \Phi\left( \begin{pmatrix} 1 & x \\ 0 & 1 \end{pmatrix} g \right) dx = 0$$ for all $g \in G(\A)$.
\end{enumerate}
\end{Def}

Note that the function $f_\Phi$ defined in $(iii)$ is well defined by $(ii)$.
We define $S_2(G,K)=\mathcal{A}_0(G,\hol,\underline{2})^K$ for any compact open subgroup $K \subseteq G(\A^\infty)$.
If $K=K_{0}(\n)$ for a non-zero ideal $\n\subset \mathcal{O}$, we write $S_2(G,\n)$ instead of $S_2(G,K_{0}(\n))$.
Right multiplication by an element $g\in G(\A^{\infty})$ induces a map $$\mathcal{R}(g)\colon S_2(G,K)\too S_2(G,gKg^{-1}).$$
Next, we are going to construct Eichler-Shimura homomorphisms, i.e.~maps $$ES_K\colon S_{2}(G,K)\too \mathcal{M}(K;\C):=\bigoplus_{\epsilon\colon F_{\infty}^{\ast}\to\left\{\pm 1\right\}}\mathcal{M}(K;\C)^{\epsilon}$$ for all compact open subgroups $K\subset G(\A^{\infty})$, which are equivariant with respect to the $G(\A^{\infty})$-action on both sides.

Let $\Phi\in S_2(G,K)$ and $f_{\phi}$ as in (iii) above.
For every $g\in G(\A^{\infty})$ we define a holomorphic $d$-differential form
\begin{align*}
 \omega_{\Phi}(g)=f_{\Phi}(\underline{z},g)dz_{1}\ldots dz_{d}
\end{align*}
on $X$.
Cuspidality of $\Phi$ implies that $\omega_{\Phi}$ lies in the space of fast decreasing, holomorphic $d$-differential forms $\Omega_{\fd,\hol}^{d}(X)$ (see \cite{Bo} for a precise definition).
An easy calculation using property (i) shows that $$ \gamma^{\ast}\left(\omega_{\Phi}(\gamma g)\right)=\omega_{\Phi}(g)$$ for all $\gamma\in G(F)^{+}$.
Hence, we have a map
\begin{align*}
 ES_{K}^{0}\colon S_2(G,K) &\too H^{0}(G(F)^{+},C(G(\A^{\infty})/K,\Omega_{\fd,\hol}^{d}(X)))\\
                      \Phi &\mapstoo \left[g\mapsto \omega_{\Phi}(g)\right]
\end{align*}
for every compact open subgroup $K\subset G(\A^{\infty})$.
It follows immediately that the diagram
\begin{center}
 \begin{tikzpicture}
    \path 	(0,0) 	node[name=A]{$S_2(G,K)$}
		(7,0) 	node[name=B]{$H^{0}(G(F)^{+},C(G(\A^{\infty})/K,\Omega_{\fd,\hol}^{d}(X)))$}
		(0,-2) 	node[name=C]{$S_2(G,gKg^{-1})$}
		(7,-2) 	node[name=D]{$H^{0}(G(F)^{+},C(G(\A^{\infty})/gKg^{-1},\Omega_{\fd,\hol}^{d}(X)))$};
    \draw[->] (A) -- (B) node[midway, above]{$ES_{K}^{0}$};
    \draw[->] (A) -- (C) node[midway, left]{$\mathcal{R}(g)$};
    \draw[->] (C) -- (D) node[midway, above]{$ES_{gKg^{-1}}^{0}$};
    \draw[->] (B) -- (D) node[midway, right]{$\mathcal{R}(g)$};
 \end{tikzpicture} 
\end{center}
is commutative for all $g\in G(\A^{\infty})$. 

Let $\bar{X}^{BS}$ the Borel-Serre bordification of $X$ with boundary $\partial X$ as constructed in \cite{BS}.
It is a smooth manifold with corners, which contains $X$ as an open submanifold.
The inclusion $X\subset \bar{X}^{BS}$ is a homotopy equivalence and thus, $\bar{X}^{BS}$ is contractible.
The action of $G(F)^{+}$ on $X$ can be extended to an action on $\bar{X}^{BS}$.

For a manifold with corners $M$ we let $C^{\sing}_{\bullet}(M)$ be the complex of singular chains in $M$ and $C^{\sm}_{\bullet}(M)$ the subcomplex of smooth chains.
By the Whitney Approximation Theorem (see \cite{Wh}) continuous chains can be approximated by smooth chains.
A standard argument shows that the inclusion $C^{\sm}_{\bullet}(M)\subset C^{\sing}_{\bullet}(M)$ is a quasi-isomorphism.
Therefore, we get a canonical $G(F)^{+}$-equivariant quasi-isomorphism between the complex of relative smooth chains $C^{\sm}_{\bullet}(\bar{X}^{BS},\partial X)$ and the complex of relative singular chains $C^{\sing}_{\bullet}(\bar{X}^{BS},\partial X)$.

Since $\bar{X}^{BS}$ is contractible, the long exact sequence for relative homology gives us an isomorphism between relative homology groups $H^{\sing}_{q}(\bar{X}^{BS},\partial X)$ and reduced homology groups $\tilde{H}^{\sing}_{q-1}(\partial X)$ of the boundary.
But $\tilde{H}^{\sing}_{q}(\partial X)$ vanishes unless $q=0$, in which case it can be identified canonically with $\Div_{0}(\PP(F))$.
Putting everything together, we get a $G(F)^{+}$-equivariant quasi-isomorphism
\begin{align}\label{qisom}
 C^{\sm}_{\bullet}(\bar{X}^{BS},\partial X)\too \Div_{0}(\PP(F))[1].
\end{align}
Let $\Delta_{q}$ denote the standard simplex of dimension $q$.
If $\omega$ is a fast decreasing $q$-differential form on $X$ and $f\colon \Delta_{q}\to \bar{X}^{BS}$ is a smooth map, the integral of $f^{\ast}\omega$ over $\Delta_{q}$ converges.
Therefore, we have a well-defined $G(F)^{+}$-equivariant pairing
\begin{align}\label{intpair}
 \Omega_{\fd,\hol}^{q}(X)\times C^{\sm}_{q}(\bar{X}^{BS},\partial X)\too \C.
\end{align}
Since holomorphic differential forms are closed, Stoke's theorem implies that the pairing vanishes on the image of the boundary map $$\partial\colon C^{\sm}_{q+1}(\bar{X}^{BS},\partial X)\too C^{\sm}_{q}(\bar{X}^{BS},\partial X).$$
Hence, the pairing \eqref{intpair} induces a morphism of complexes
\begin{align}\label{intmap}
 \Omega_{\fd,\hol}^{q}(X)\too \Hom(C^{\sm}_{q}(\bar{X}^{\bullet},\partial X),\C)[q].
\end{align}
Applying (hyper) group cohomology, we get maps
\begin{align*}
				 &H^{0}(G(F)^{+},C(G(\A^{\infty})/K,\Omega_{\fd,\hol}^{q}(X))) \\
 \stackrel{\eqref{intmap}}{\too} &H^{q}(G(F)^{+},C(G(\A^{\infty})/K,\Hom(C^{\sm}_{\bullet}(\bar{X}^{BS},\partial X),\C))) \\
  \stackrel{\eqref{qisom}}{\too} &H^{q-1}(G(F)^{+},\Ah(K;N)).
\end{align*}
for all open compact subgroups $K\subset G(\A^{\infty})$ (see \eqref{Ah} at the beginning of Section \ref{cohomology} for the definition of $\Ah(K;N)$).

Precomposing the just constructed map for $q=d$ with the map $ES_{K}^{0}$ above we get the Eichler-Shimura homomorphism
\begin{align*}
 ES_{K}\colon S_{2}(G,K)\too H^{d-1}(G(F)^{+},\Ah(K;N))\cong \mathcal{M}(K;\C).
\end{align*}
Here the last isomorphism is given by Shapiro's Lemma.

Chasing through the definitions we easily see that the diagram
\begin{center}
 \begin{tikzpicture}
    \path 	(0,0) 	node[name=A]{$S_2(G,K)$}
		(4,0) 	node[name=B]{$\mathcal{M}(K;\C)$}
		(0,-2) 	node[name=C]{$S_2(G,gKg^{-1})$}
		(4,-2) 	node[name=D]{$\mathcal{M}(K;\C)$};
    \draw[->] (A) -- (B) node[midway, above]{$ES_{K}$};
    \draw[->] (A) -- (C) node[midway, left]{$\mathcal{R}(g)$};
    \draw[->] (C) -- (D) node[midway, above]{$ES_{gKg^{-1}}$};
    \draw[->] (B) -- (D) node[midway, right]{$\mathcal{R}(g)$};
 \end{tikzpicture} 
\end{center}
is commutative for all compact open subgroups $K\subset G(\A^{\infty})$ and all $g\in G(\A^{\infty})$.
Especially, if $K=K_{0}(\n)$ for a non-zero ideal $\n\subseteq \mathcal{O}$, the map $ES_{\n}=ES_{K_0(\n)}$ is equivariant with respect to the action of the Hecke and Atkin-Lehner operators on both sides.

%% file: secaut-specval.tex
\subsection{Automorphic Stickelberger elements and special values}
In the following we will give the relation of Stickelberger elements to special values of L-functions.

Let us fix two non-zero ideals $\n,\m\subset \mathcal{O}$ and a continuous character $\epsilon\colon F_{\infty}^{\ast}\to \left\{\pm 1\right\}$.
Further, we fix a finite extension $L/F$ with Galois group $\mathcal{G}$.
We assume that $U(\m)$ is contained in the kernel of the Artin reciprocity map $\I\to \mathcal{G}$.
For a cuspidal automorphic form $\Phi\in S^{2}(G,\n)$ we write $\kappa_{\Phi}^{\epsilon}$ for the image of $ES_{\n}(\Phi)$ under the projection
$$\mathcal{M}(\n;\C)\too \mathcal{M}(\n;\C)^{\epsilon}.$$

\begin{Def}
The Stickelberger element of modulus $\mathfrak{m}$ and sign $\epsilon$ associated to $\Phi\in S_{2}(G,\n)$ and $L/K$ is defined as
\begin{align*}
 \Theta_{\mathfrak{m}}(L/F,\Phi)^{\epsilon}=\Theta_{\mathfrak{m}}(L/F,\kappa_{\Phi}^{\epsilon})\in \C[\G].
\end{align*}
\end{Def}

We write $\mathbb{T}$ for the infinite polynomial ring $\Z[T_{\p}]$ generated by indeterminants $T_{\p}$ for all finite places of $F$.
Let $\pi=\otimes\pi_{v}$ be a cuspidal automorphic representation of $G$ such that the components $\pi_v$ at all non-Archimedean places $v$ of $F$ are discrete series of weight $2$.
By Theorem 1 of \cite{Ca} there exists a unique non-zero ideal $\mathfrak{f}(\pi)$ of $\mathcal{O}$ such that
\begin{align*}
 \dim_{\C}(\otimes_{v\nmid\infty}\pi_{v})^{K_{0}(\mathfrak{f}(\pi))}=1.
\end{align*}
It is called the conductor of $\pi$.
We will assume from now on that $\n=\mathfrak{f}(\pi)$.
The Hecke operators act by scalars on $(\otimes_{v\nmid\infty}\pi_{v})^{K_{0}(\n)}$ and thus determine a ring homomorphism $\chi_{\pi}\colon\mathbb{T}\to\C$.
In fact, there is a finite extension $\Q_{\pi}$ of $\Q$ (the field of definition of $\pi$) with ring of integers $\mathcal{O}_{\pi}$ such that $\chi_{\pi}$ factors over $\mathcal{O}_{\pi}$.
Let $\mathfrak{a}_{\pi}$ be the kernel of $\chi_{\pi}$.
Given any $\mathbb{T}$-module $M$ we write $M[\pi]=\Hom_{\mathbb{T}}(\mathbb{T}/\mathfrak{a}_{\pi},M)$ for the $\chi_{\pi}$-isotypical part of $M$.
By the multiplicity one theorem we have
\begin{align*}
 \dim_{\C}S_{2}(G,\n)[\pi]=1.
\end{align*}
The non-zero elements of $S^{2}(G,\n)[\pi]$ are called newforms of $\pi$.
Since newforms are Hecke eigenvectors we can apply Proposition \ref{compatibility} to get relations between the Stickelberger elements of different moduli.
To characterize $\Theta_{\mathfrak{m}}(L/F,\Phi)^{\epsilon}$ for a given newform $\phi$ uniquely we have to evaluate these Stickelberger elements on characters of conductor $\m$.

In the following the Haar measure $dx=\prod_v dx_v$ on $\I$ is normalized such that $\vol(U_\p,dx_\p)=1$ for all finite places $\p$ of $F$ and we fix a non-zero character $\psi\colon F\backslash\A\to \C^{\ast}$.
Given a character $\chi \colon \G \to \C$ we denote by $L^{(\m)}(s,\pi\otimes\chi)$ the standard $L$-function of $\pi \otimes \chi$ without the Euler factors at primes diving $\m$.

\begin{Proposition}\label{specialvalues}
Let $\Phi$ be a newform of some $\pi$ as above.
There exists a constant $c\in \C^{\ast}$, which is independent of $\m$, such that for all characters $\chi\colon \I\to\G\to \C^{\ast}$ of conductor $\m$ with $\chi_{\infty}=\epsilon$ we have
\begin{align*}
 \chi(\Theta_{\mathfrak{m}}(L/F,\Phi)^{\epsilon})=c\ \tau(\chi^{-1})L^{(\m)}(1/2,\pi\otimes\chi).
\end{align*}
Here $\tau(\chi^{-1})=\tau(\chi^{-1},\psi,dx)$ is the Gauss sum of $\chi$ with respect to our choice of a Haar measure $dx$ on $\I$ and an additive character $\psi$.
\end{Proposition}

\begin{proof}
By a standard argument (see for example \cite{Ha}) there exists a non-zero constant $c \in \C^{\ast}$ (independent of $\m$ and $\chi$) such that
\begin{align*}
 \chi(\Theta_{\mathfrak{m}}(L/F,\Phi)^{\epsilon})=c\ [U :U(\m)]\cdot\int_{F^{\ast}\backslash\I}\Phi\left( \begin{pmatrix} x & 0 \\ 0 & 1 \end{pmatrix} g_{\m} \right) \chi(x) dx,
\end{align*}
where $g_{\m}$ is the matrix defined in \eqref{matrix}.
For $s\in \C$ the integral
\begin{align*}
 \int_{F^{\ast}\backslash\I}\Phi\left( \begin{pmatrix} x & 0 \\ 0 & 1 \end{pmatrix} g_{\m} \right) \chi(x) |x|^{s} dx
\end{align*}
defines a holomorphic function.
Let $W$ denote the $\psi$-Whittaker function of $\mathcal{R}(g_\m) \Phi$.
Since $\mathcal{R}(g_\m) \Phi \in \pi $ is a pure tensor, we can factor $W$ as a product of local Whittaker functions $W_{v}$.
For $\Re(s)$ large we can unfold the above integral to get
\begin{align*}
 \int_{F^{\ast}\backslash\I}\Phi\left( \begin{pmatrix} x & 0 \\ 0 & 1 \end{pmatrix} g_{\m} \right) \chi(x) |x|^{s} dx &=\int_{\I} W\left( \begin{pmatrix} x & 0 \\ 0 & 1 \end{pmatrix} \right) \chi(x) dx \\
														      &=\prod_{v} \int_{F_{v}} W_v\left( \begin{pmatrix} x_v & 0 \\ 0 & 1 \end{pmatrix} \right) \chi_v(x_v) dx_v.
\end{align*} 
Therefore, we are reduced to a computation of local integrals, which we will carry out in the rest of this section.
\end{proof}

Let $\p$ be a finite place of $F$ and $\pi_{\p}$ an infinite dimensional, irreducible, smooth representations of $G(F_\p)$ of conductor $\p^{n}$, i.e.~we have
\begin{align*}
 \dim_{\C}\pi_{\p}^{K_\p(\p^{n})}=1.
\end{align*}
The non-zero elements of $\pi_{\p}^{K_\p(\p^{n})}$ are called local newforms.
Let $\Lambda$ be a $\psi_{\p}$-Whittaker functional of $\pi_{\p}$.
By definition $\Lambda$ is a non-zero linear functional on $\pi_{\p}$ such that
\begin{align*}
 \Lambda\left(\begin{pmatrix}1&x\\ 0&1\end{pmatrix}\varphi\right)=\psi_{\p}(x)\Lambda(\varphi)
\end{align*}
for all $\varphi\in\pi_{\p}$ and all $x\in F_{\p}$.

\begin{Lemma}
Let $\varphi\in\pi_{\p}$ be a local newform.
For every character $\chi_{\p}\colon F_{\p}^{\ast}\to\C^{\ast}$ of conductor $\p^{m}$ the following integral converges for $\Re(s)$ large and we have an equality
\begin{align*}
  & [U_\p : U_{\p}^{(m)}]\int_{F_{\p}^{\ast}}\Lambda\left( \begin{pmatrix} x&0\\0&1 \end{pmatrix} \begin{pmatrix} \varpi_{\p}^{m}&1\\0&1 \end{pmatrix} \varphi\right) \chi_{\p}(x) |x|^{s}_\p dx \\
 =& c\cdot \tau(\chi_{\p}^{-1},\psi_{\p}) N(\p)^{(t+m)s}L^{(\p^{m})}(s+1/2,\pi_{\p}\otimes\chi_{\p}),
\end{align*}
where $\varpi_{\p}$ is a local uniformizer at $\p$, $c\in \C$ and $t\in \Z$ are constants independent of $\chi_{\p}$ and $m$ and $$L^{(\p^{m})}(s,\pi_{\p} \otimes \chi_{\p}) = \begin{cases}
                                                                                                                                                                                 L(s,\pi_{\p}\otimes\chi_{\p})  & \mbox{ if $m=0$,} \\
                                                                                                                                                                                 1				& \mbox{ if $m>0$.}
                                                                                                                                                                                \end{cases} $$
\end{Lemma}
\begin{proof}
Let $\p^{-t}$ be the conductor of $\psi_{\p}$.
A straightforward calculation shows that $$\Lambda^{\prime}(\varphi)=\Lambda\left(\begin{pmatrix} \varpi_{\p}^{-t} & 0 \\ 0 & 1 \end{pmatrix} \varphi \right)$$ defines a Whittaker functional with respect to an additive character $\psi^{\prime}$ of conductor $\mathcal{O}$.
It is well known that
\begin{align*}
\Lambda^{\prime}\left(\begin{pmatrix} x & 0 \\ 0 & 1 \end{pmatrix} \varphi\right) =0
\end{align*}
if $\ord_{\p}(x)< 0$  and equal to a non-zero complex number $c\in \C^{\ast}$ for $\ord_{\p}(x)=0$ (see for example \cite{Mi}).
Without loss of generality we may assume that $c=1$. 
Hence, for $\Re(s)$ large we have the following equality
\begin{align*}
  & \int_{F_{\p}^{\ast}}\Lambda\left( \begin{pmatrix} x&0\\0&1 \end{pmatrix} \begin{pmatrix} \varpi_{\p}^{m}&1\\0&1 \end{pmatrix} \varphi\right) \chi_{\p}(x)  |x|^{s}_\p dx \\
 =& \int_{F_{\p}^{\ast}}\Lambda^{\prime}\left( \begin{pmatrix} x\varpi_{\p}^{t}&0\\0&1 \end{pmatrix} \begin{pmatrix} \varpi_{\p}^{m}&1\\0&1 \end{pmatrix} \varphi\right) \chi_{\p}(x) |x|^{s}_\p dx \\
 =& \int_{F_{\p}^{\ast}}\Lambda^{\prime}\left( \begin{pmatrix} 1&x \varpi_{\p}^{t}\\0&1 \end{pmatrix} \begin{pmatrix} x\varpi_{\p}^{m+t}&0\\0&1 \end{pmatrix} \varphi\right) \chi_{\p}(x)|x|^{s}_\p dx \\
 =& \int_{F_{\p}^{\ast}}\Lambda^{\prime}\left( \begin{pmatrix} x\varpi_{\p}^{m+t}&0\\0&1 \end{pmatrix} \varphi\right) \psi^{\prime}(x\varpi_{\p}^{t}) \chi_{\p}(x)|x|^{s}_\p dx \\
 =& \chi_{\p}(\varpi_{\p}^{-t})|\varpi_{\p}^{-t}|_p^{s}\sum_{k=0}^{\infty} \Lambda^{\prime} \left(\begin{pmatrix}\varpi_{\p}^{k}& 0 \\ 0 & 1\end{pmatrix}\varphi \right) \int_{\varpi_{\p}^{k-m}U_\p} \psi^{\prime}(x) \chi_{\p}(x)|x|^{s}_\p dx.
\end{align*}
By classical formulas for the Whittaker functional of a newform (see for example \cite{Mi}) we have
\begin{align*}
 \Lambda^{\prime} \left(\begin{pmatrix}\varpi_{\p}^{k}& 0 \\ 0 & 1\end{pmatrix}\varphi \right) = |\varpi_{\p}^{k}|_{\p}^{1/2}\sum_{r+s=k,\ r,s\geq 0} \alpha_{1}^{r}\alpha_{2}^{s},
\end{align*}
where $\alpha_{i}\in \C$, $1\leq i \leq 2$, are the complex number such that $$L(s,\pi_{\p})=\prod_{i=1}^{2}(1-\alpha_i |\varpi_{\p}|_{\p}^{s})^{-1}.$$
Therefore, if $m=0$, we obtain
\begin{align*}
  & \chi_{\p}(\varpi_{\p}^{-t})|\varpi_{\p}^{-t}|_{\p}^{s} \sum_{k=0}^{\infty} \Lambda^{\prime} \left(\begin{pmatrix}\varpi_{\p}^{k}& 0 \\ 0 & 1\end{pmatrix}\varphi \right) \int_{\varpi_{\p}^{k}U_\p} \psi^{\prime}(x) \chi_{\p}(x)|x|^{s}_\p dx \\
 =& \tau(\chi_{\p}^{-1},\psi_{\p}) N(\p)^{ts} \sum_{k=0}^{\infty} \chi_{\p}(\varpi_{\p}^{k})|\varpi_{\p}^{k}|_{\p}^{s+1/2}\left(\sum_{r+s=k,\ r,s\geq 0} \alpha_{1}^{r}\alpha_{2}^{s}\right) \\
 =& \tau(\chi_{\p}^{-1},\psi_{\p}) N(\p)^{ts} \prod_{i=1}^{2}(1-\alpha_i \chi_{\p}(\varpi_{\p})|\varpi_{\p}|_{\p}^{s+1/2})^{-1} \\
 =& \tau(\chi_{\p}^{-1},\psi_{\p}) N(\p)^{ts} L(s+1/2,\pi_{\p}\otimes\chi_{\p}).
\end{align*}
In the case $m\geq 1$ we can use Lemma 2.2 of \cite{Sp} to get
\begin{align*}
  & \chi_{\p}(\varpi_{\p}^{-t})|\varpi_{\p}^{-t}|_p^{s} \sum_{k=0}^{\infty} \Lambda^{\prime} \left(\begin{pmatrix}\varpi_{\p}^{k}& 0 \\ 0 & 1\end{pmatrix}\varphi \right) \int_{\varpi_{\p}^{k-m}U_\p} \psi^{\prime}(x) \chi_{\p}(x)|x|^{s}_\p dx \\
 =& \chi_{\p}(\varpi_{\p}^{-t})|\varpi_{\p}^{-t}|_p^{s} \int_{\varpi_{\p}^{-m}U_{\p}} \psi^{\prime}(x) \chi_{\p}(x)|x|^{s}_\p dx \\
 =& [U_\p : U_\p^{(m)}]^{-1} \tau(\chi_{\p}^{-1},\psi_{\p}) N(\p)^{(t+m)s}
\end{align*}
and thus the claim follows.
\end{proof}

\begin{Remark}
If $\pi_{\p}$ is an unramified principal series, an unramified twist of the Steinberg representation or a supercuspidal representation and $\chi_{\p}\colon F_{\p}^{\ast}\to \C^{\ast}$ is a character of conductor $\p^{m}$ with $m\geq 1$, then the local Euler factor $L(s,\pi_{\p}\otimes\chi_{\p})$ is equal to $1$.
\end{Remark}

%% file: secaut-lifting.tex
\subsection{Integrality and special lifts}
We show that Stickelberger elements of Hilbert modular cusp forms can be defined integrally and use the results of the first chapter to bound their order of vanishing.
As in the previous section let $\pi=\otimes\pi_{v}$ be a cuspidal automorphic representation of $G$ of conductor $\n$ such that the components $\pi_v$ at all non-Archimedean places $v$ of $F$ are discrete series of weight $2$.
Let $\Q_{\pi}$ be the field of definition of $\pi$ with ring of integers $\mathcal{O}_{\pi}$.
In the following we fix a character $\epsilon\colon F_{\infty}^{\ast}\to \left\{\pm 1\right\}$. 

\begin{Def}
An automorphic cusp form $\Phi\in S_{2}(G,\n)[\pi]$ is called integral if $ES_{\pi}^{\epsilon}(\Phi)$ lies in the image of the map
\begin{align*}
 M(\n;\mathcal{O}_{\pi})^{\epsilon}\too M(\n;\C)^{\epsilon}.
\end{align*}
We let $S_{2}(G,\n;\mathcal{O}_{\pi})_\pi$ be the $\mathcal{O}_{\pi}$-module of integral forms in $S_{2}(G,\n)[\pi]$.
More generally, let $\mathcal{L}\subset \C$ be a free $\mathcal{O}_{\pi}$-submodule of rank $1$.
We define $S_{2}(G,\n;\mathcal{L})_\pi$ as the submodule of all forms in $S_{2}(G,\n)[\pi]$ whose modular symbols lie in the image of the map
\begin{align*}
 M(\n;\mathcal{L})^{\epsilon}\too M(\n;\C)^{\epsilon}.
\end{align*}
\end{Def}

\begin{Lemma}\label{intStick} \renewcommand{\labelenumi}{(\roman{enumi})}
\begin{enumerate}
\item The $\mathcal{O}_{\pi}$-module $S_{2}(G,\n;\mathcal{O}_{\pi})_\pi$ is locally free of rank 1.
\item Let $\Phi\in S_{2}(G,\n)[\pi]$.
There exists a free $\mathcal{O}_{\pi}$-submodule $\mathcal{L}\subset \C$ of rank $1$ such that $\Phi \in S_{2}(G,\n;\mathcal{L})_\pi$.
Given any such submodule $\mathcal{L}\subset \C$ we have $\Theta_{\mathfrak{m}}(L/F,\Phi)^{\epsilon}\in \Z[\mathcal{G}]\otimes \mathcal{L}$ for all non-zero ideals $\m\subseteq\mathcal{O}$ and all finite abelian extensions $L/F$ with ramification bounded by $\m$. 
\end{enumerate}
\end{Lemma}

\begin{proof}
(i) By a theorem Eichler-Shimura and Harder (see for example \cite{Ha}) the space $M(\n,\C)^{\epsilon}[\pi]$ is one-dimensional for every continuous character $\epsilon\colon F^{\ast}\to \left\{\pm 1\right\}$ and the Eichler-Shimura homomorphism $ES_{\n}$ constructed in Section \ref{automorphic} induces an isomorphism
\begin{align*}
 ES_{\pi}^{\epsilon}\colon S_{2}(G,\n)[\pi]\stackrel{\cong}{\too}\mathcal{M}(\n;\C)^{\epsilon}[\pi].
\end{align*}
Proposition \ref{FlachundNoethersch} implies that the map
\begin{align*}
 M(\n;\Q_{\pi})^{\epsilon}[\pi] \otimes_{\Q_{\pi}}\C\to M(\n;\C)^{\epsilon}[\pi]
\end{align*}
is an isomorphism. Using Proposition \ref{FlachundNoethersch} again we see that the intersection of the image of $M(\n;\mathcal{O}_{\pi})^{\epsilon}$ in $M(\n;\Q_{\pi})^{\epsilon}$ with $M(\n;\Q_{\pi})^{\epsilon}[\pi]$ is locally free of rank $1$.\\
(ii) The existence of $\mathcal{L}$ follows from (i).
The second statement follows directly from the definitions.
\end{proof}

Let $S=S_{\pi}$ be the set of finite places $\q$ of $F$ such that $\pi_{\q}$ is the Steinberg representation.

It is well known that every $\q \in S$ divides the conductor $\n$ of $\pi$ exactly once.
For every subset $S^{\prime}\subseteq S$ let $\mathfrak{a}_{S^{\prime}}^{\tor}\subseteq \mathcal{O_{\pi}}$ be the product of the annihilator of the $2$-torsion subgroup of $M(\n;\mathcal{O}_{\pi})^{\epsilon}$ and the annihilator of the torsion subgroup of
\begin{align*}
\bigoplus_{q\in S^{\prime}} M(\n\q^{-1};\mathcal{O}_{\pi})^{\epsilon}.
\end{align*} 
If $F\neq \Q$ we define $c_{S^{\prime}}=\gcd(\prod_{\q\in S^{\prime\prime}}(N(\q)+1)\mid S^{\prime\prime}\subset S^{\prime} \mbox{ with } \left|S^{\prime}\right|=\left|S^{\prime\prime}\right|+1)$. In the case $F=\Q$ we simply put $c_{S^{\prime}}=1$. Finally, we define $\mathfrak{a}_{S^{\prime}}=c_{S^{\prime}}\cdot\mathfrak{a}_{S^{\prime}}^{\tor}$

\begin{Remark}\label{allmostfinal} \renewcommand{\labelenumi}{(\roman{enumi})}
\begin{enumerate}
\item If $F=\Q$, we are only dealing with cohomology groups in degree $0$, which are always torsion free.
Thus, in this case $\mathfrak{a}_{S^{\prime}}=\mathcal{O}_{\pi}$.
\item Suppose $\pi$ corresponds to a modular elliptic curve $A$ over $F$.
Then its field of definition is $\Q$ and $S_{\pi}$ is exactly the set of primes $\q$ such that $E$ has split multiplicative reduction at $\q$.
\end{enumerate}
\end{Remark}

\begin{Lemma}\label{torsion}
Let $\mathcal{L}\subset \C$ be a a free $\mathcal{O}_{\pi}$-submodule of rank $1$ and $\Phi \in S_{2}(G,\n;\mathcal{L})_\pi$ an automorphic form.
Then $r\cdot ES_\n(\Phi)$ lies in the image of the map
\begin{align*}\label{liftsequence}
 M(\n,S^{\prime};\mathcal{L})^{\epsilon}\too M(\n;\mathcal{L})^{\epsilon}\too M(\n;\C)^{\epsilon}
\end{align*}
for all $r\in \mathfrak{a}_{S^{\prime}}$ and all subsets $S^{\prime}\subseteq S$.
\end{Lemma}

\begin{proof}
Let $\kappa$ be a lift of $ES_\n(\Phi)$ to $M(\n;\mathcal{L})^{\epsilon}$.
Since $\pi$ is new at $\q$ for all $\q\in S$ we know that the image of $\kappa$ under the trace maps
$$\Tr_\q\colon  M(\n;\mathcal{L})^{\epsilon}\too M(\n\q^{-1};\mathcal{L})^{\epsilon}$$
is torsion for all $\q\in S$.
Thus, every multiple $t\cdot\kappa$ with $t\in \mathfrak{a}_{S^{\prime}}^{\tor}$ lies in the kernel of all trace maps $\Tr_\q$ for $\q\in S^{\prime}$ and is also an eigenvector under the Atkin-Lehner operator $W_\q$ with eigenvalue $-1$.

We are going to prove the claim by induction on $\left|S^{\prime}\right|$.
The well-known short exact sequence
$$0\too \cind_{K_\q(\q)}^{G(F_\q)}\Z\too \left(\cind_{K_\q}^{G(F_\q)}\Z\right)^{W_\q=-1}\too \St_\q\too 0$$
(see for example Section 2.4 of \cite{Sp})
induces an exact sequence
$$M(\n,\{\q\};\mathcal{L})^{\epsilon}\xrightarrow{\Ev_{\{\q\}}} \left(M(\n;\mathcal{L})^{\epsilon}\right)^{W_\q=-1}\xrightarrow{\Tr_\q} M(\n\q^{-1};\mathcal{L})^{\epsilon}$$ in cohomology.
Hence, we can lift the class $t\cdot\kappa$ if $S^{\prime}=\{\q\}$.

Now, let $\left|S^{\prime}\right|\geq 2$.
We pick an element $\q\in S^{\prime}$ and consider the following commutative diagram with exact rows:
\begin{center}
 \begin{tikzpicture}
    \path 	(-0.5,0) 	node[name=A]{$M(\n,S^{\prime};\mathcal{L})^{\epsilon}$}
		(4,0) 	node[name=B]{$\left(M(\n,S^{\prime}-\{\q\};\mathcal{L})^{\epsilon}\right)^{W_\q=-1}$}
		(9,0) 	node[name=C]{$M(\n\q^{-1},S^{\prime}-\{\q\};\mathcal{L})^{\epsilon}$}
		(-0.5,-2) 	node[name=D]{$M(\n,\{\q\};\mathcal{L})^{\epsilon}$}
		(4,-2) 	node[name=E]{$\left(M(\n;\mathcal{L})^{\epsilon}\right)^{W_\q=-1}$}
		(9,-2)  node[name=F]{$M(\n\q^{-1};\mathcal{L})^{\epsilon}$};
    \draw[->] (B) -- (C) node[midway, above]{$\Tr_\q$};
    \draw[->] (E) -- (F) node[midway, above]{$\Tr_\q$};
		\draw[->] (A) -- (B) node[midway, above]{$\Ev_{\{\q\}}$};
		\draw[->] (D) -- (E) node[midway, above]{$\Ev_{\{\q\}}$};
		\draw[->] (A) -- (D) node[midway, right]{$\Ev_{S^{\prime}-\{\q\}}$};
		\draw[->] (B) -- (E) node[midway, right]{$\Ev_{S^{\prime}-\{\q\}}$};
		\draw[->] (C) -- (F) node[midway, right]{$\Ev_{S^{\prime}-\{\q\}}$};
 \end{tikzpicture} 
\end{center}
By the induction hypothesis we can lift $r\cdot \kappa$ to a class $\widetilde{\kappa}\in \left(M(\n,S^{\prime}-\{\q\};\mathcal{L})^{\epsilon}\right)^{W_\q=-1}$ for every $r\in \mathfrak{a}_{S^{\prime}-\{\q\}}$.
If $F=\Q$, the map
$$\Ev_{S^{\prime}-\{\q\}}\colon M(\n\q^{-1},S^{\prime}-\{\q\};\mathcal{L})^{\epsilon}\too M(\n\q^{-1};\mathcal{L})^{\epsilon}$$
is injective.
Therefore, the claim follows from $\Tr_\q(r\cdot\kappa)=0$.

In all other cases, we consider the canonical map
\begin{align*}
\iota\colon M(\n\q^{-1},S^{\prime}-\{\q\};\mathcal{L})^{\epsilon} \too M(\n,S^{\prime}-\{\q\};\mathcal{L})^{\epsilon}.
\end{align*}
Since $\Tr_\q\circ\iota$ is equal to multiplication by $N(\q)+1$, we see that
$(N(\q)+1)\cdot\widetilde{\kappa}-\iota(\Tr_\q(\widetilde{\kappa}))$ can be lifted to a class in $M(\n,S^{\prime};\mathcal{L})^{\epsilon}$. 
The commutativity of the above diagram together with the fact that $\Tr_\q(r\cdot \kappa)$ vanishes implies that $(N(\q)+1)\cdot\widetilde{\kappa}-\iota(\Tr_\q(\widetilde{\kappa}))$ is also a lift of $(N(\q)+1)r\cdot\kappa$.
\end{proof}

As in Section \ref{Stick} let $L/F$ be a finite abelian extension with Galois group $\mathcal{G}$.
Suppose that the ramification of $L/F$ is bounded by the non-zero ideal $\mathfrak{m}\subset \mathcal{O}$.
We let $S_{\m}$ be the subset of primes in $S$ which divide $\m$.
If $\q \in S$, we define $I_{\q} \subseteq \mathcal{R}[\mathcal{G}]$ as the kernel of the projection $\mathcal{R}[\mathcal{G}]\onto \mathcal{R}[\mathcal{G}/\mathcal{G}_{\q}]$.
If $v\in S_{\infty}$, we let $\sigma_{v}$ be a generator of $\mathcal{G}_{v}$ and define $I_{v}^{\pm 1}\subseteq \mathcal{R}[\mathcal{G}]$ as the ideal generated by $\sigma_{v}\mp 1$.
\begin{Theorem}\label{MainTheorem}
Let $\Phi\in S_{2}(G,\n;\mathcal{L})_{\pi}$. Then $$r\cdot \Theta_{\mathfrak{m}}(L/F,\Phi)^{\epsilon}\in  \left(\prod_{v\in S_{\infty}} I_{v}^{-\epsilon_{v}(-1)} \cdot \prod_{\q \in S_\mathfrak{m}} I_{\q}\right) \otimes \mathcal{L},$$ for all $r\in \mathfrak{a}_{S_{m}}$.
In particular, if $L/F$ is totally real, $\epsilon$ is the trivial character and $\mathcal{R}\subset \C$ is a ring containing $\mathcal{L}$, we get
\begin{align*}
 2^{-d}\cdot\Theta_{\mathfrak{m}}(L/F,\Phi)^{\epsilon}&\in \mathcal{R}[\mathcal{G}] \intertext{and} \ord_{\mathcal{R}}(2^{-d}r\cdot\Theta_{\mathfrak{m}}(L/F,\Phi)^{\epsilon})&\geq |S_{\m}|.
\end{align*}
\end{Theorem}

\begin{proof}
This is a direct consequence of Lemma \ref{torsion} and Proposition \ref{vanish2}.
\end{proof}

\begin{Remark} \renewcommand{\labelenumi}{(\roman{enumi})}
\begin{enumerate}
\item Since newforms are also eigenvectors for all Atkin-Lehner operators, we can (after inverting 2) use Corollary \ref{parity} to calculate the parity of the order of vanishing of the above Stickelberger elements in terms of local root numbers.
Hence, if $(-1)^{S_{\m}}$ and the sign in the functional equation differ, we can deduce that the order of vanishing is at least $|S_{\m}|+1$.
\item If $\Phi$ is a $p$-ordinary newform we can apply Remark \ref{padic} (ii) to the ordinary $p$-stabilization of $\Phi$ to construct its $p$-adic $L$-function $L_{p}(s,\Phi)$.
Using a version of Theorem \ref{MainTheorem} for the $p$-stabilization of $\Phi$ and observing that the ideal $\mathfrak{a}^{\tor}_{S_p}$ is independent of the modulus, we can deduce that the order of vanishing  of $L_{p}(s,\Phi)$ at $s=0$ is at least the number of primes lying above $p$ at which the automorphic representation associated to $\Phi$ is Steinberg.
This was first proven by Spie\ss~ in \cite{Sp}.
\end{enumerate}
\end{Remark}

Finally, we want to spend a few words on how the theorem of the introduction can be obtained as a special case.
Let $M>2$ be a natural number and $A$ an elliptic curve over $\Q$ with corresponding normalized newform $f\in S_{2}(\Gamma_{0}(N))$.
Let $+\colon \R^{\ast}\to \left\{\pm 1\right\}$ be the trivial character.
Comparing Proposition \ref{compatibility} and Proposition \ref{specialvalues} with the corresponding statements for the Stickelberger elements of Mazur and Tate shows that there exists a constant $c\in \C^{\ast}$ such that
\begin{align*}
 (\Theta^{\MT}_{A,M})^{\vee}=c\cdot \Theta_{M}\left(\Q(\mu_{M})^{+}/\Q, \frac{2 \pi i}{\Omega_A^+}\ f \right)^{+}.
\end{align*}
A direct calculation shows that $c=2$.
The modular symbol $ES_{\pi}^{\epsilon}\left(\frac{2 \pi i}{\Omega_A^+} f\right)$ of $f$ is contained in $M(N,\mathcal{L})^{\epsilon}$ if and only if $[q]_{A}\in \mathcal{L}$ for all $q\in \Q/\Z$.
Thus, our main theorem follows from Remark \ref{allmostfinal} and Theorem \ref{MainTheorem}.